\documentclass[a4paper]{amsart}
\usepackage[all]{xy}

\def\AA{{\mathbb{A}}}

\def\RR{{\mathbb{R}}}
\def\CC{{\mathbb{C}}}
\def\QQ{{\mathbb{Q}}}
\def\NN{{\mathbb{N}}}
\def\ZZ{{\mathbb{Z}}}

\let \cedilla =\c

\renewcommand{\a}{{\frak{a}}}
\renewcommand{\o}{{\mathcal O}}
\newcommand{\mld}{{\mathrm{mld}}}
\newcommand{\mldmj}{{\mathrm{mld_{MJ}}}}

\renewcommand{\j}{{\mathcal{J}}}
\newcommand{\cont}{{\mathrm{Cont}}}
\newcommand{\codim}{{\mathrm{codim}}}
\newcommand{\ord}{{\mathrm{ord}}}
\newcommand{\mult}{{\mathrm{mult}}}
\newcommand{\spec}{{\mathrm{Spec \ }}}

\newcommand{\hk}{{\widehat{k}}}

\newcommand{\val}{{\mathrm{val}}}

\newcommand{\amj}{{a_{\operatorname{MJ}}}}
\newcommand{\sing}{{\mathrm{Sing}}}
\newcommand{\emb}{{\mathrm{emb}}}

\newcommand{\Fitt}{{\operatorname{Fitt}}}

\newcommand{\pp}{{\bf p}}
\newcommand{\bx}{{\bf x}}
\newcommand{\mm}{{\bf m}}
\newcommand{\height}{{\mathrm{ht}}}
\newcommand{\ba}{{\bold a}}
\newcommand{\bb}{{\bold b}}
\newcommand{\ini}{{\mathrm{in}}}
\newcommand{\he}{{\widetilde E}}
\newcommand{\hs}{{hypersurface}}

\newcommand{\resp}{{resp.}}
\newcommand{\lla}{{\langle\langle}}
\newcommand{\rra}{{\rangle\rangle}}

\newtheorem{thm}{Theorem}[section]
\newtheorem{cor}[thm]{Corollary}
\newtheorem{Corollary-Definition}[thm]{Corollary-Definition}
\newtheorem{prop-def}[thm]{Proposition-Definition}
\newtheorem{prop}[thm]{Proposition}
\newtheorem{lem}[thm]{Lemma}

\newtheorem{conj}[thm]{Conjecture}

\theoremstyle{definition}
\newtheorem{defn}[thm]{Definition}
\newtheorem{exmp}[thm]{Example}
\newtheorem{rem}[thm]{Remark}

\newtheorem{Proposition-Definition}[thm]{Proposition-Definition}

\begin{document}

\title[Mather-Jacobian-log discrepancies]{Finite determination conjecture \\ for Mather-Jacobian minimal  log discrepancies\\
and its applications}

\thanks{
The author is partially supported by Grant-In-Aid (c) 1605089 of JSPS in Japan.}

\author{Shihoko Ishii }
\maketitle
\begin{abstract} In this paper we study singularities in arbitrary characteristic.
We propose Finite Determination Conjecture for Mather-Jacobian minimal log discrepancies in terms of jet schemes of a singularity.
The conjecture is equivalent to 
  the boundedness of the number of the blow-ups to obtain a prime divisor which computes the Mather-Jacobian minimal log
discrepancy.
We also show that this conjecture yields some basic properties of singularities;
{\it eg.,} openness of Mather-Jacobian (log) canonical singularities, 
stability of these 
singularities under small deformations and 
lower semi-continuity of Mather-Jacobian minimal log discrepancies,  which are already  known in characteristic 0 and
open for positive characteristic case.
We show some evidences of the conjecture:
 for 
example, for
non-degenerate \hs  \ of any dimension in arbitrary characteristic 
and 
 $2$-dimensional singularities in characteristic not $2$. 
  We also give 
  a bound of the number of the blow-ups to obtain a prime divisor which
 computes the Mather-Jacobian minimal log discrepancy.

\end{abstract}
\section{Introduction}
\noindent
Studies of singularities with respect to ``discrepancies"  on a variety over a field of characteristic $0$ are developed 
based on resolutions of singularities, generic smoothness (Strong Bertini  Theorem)
and vanishing theorems of cohomologies of Kodaira type.
However these are not available for varieties over a field of positive characteristic.
This is the reason why the study  of singularities in positive characteristic case did not develop 
in the same direction.

On the other hand, 
a singularity in positive characteristic has been studied from a different view point; in terms of
Frobenius map which is specific for positive characteristic case.
Then,
a surprising correspondence between  singularities in characteristic zero with respect to ``discrepancies" and singularities in positive characteristic with respect to Frobenius map started to be unveiled by  N.  Hara and K-i. Watanabe \cite{hw}, and then many beautiful 
results in this direction are discovered by the contributions of many   people.
We do not cite all references about these results here, because it is not the main theme of this paper.

The standing point of this paper is apart from theirs. 
We will try to study singularities in positive characteristic in the same line as in the characteristic
0 case.
More precisely, we study singularities in terms of ``discrepancies" which are common for 
any characteristic.
For this sake, lack of resolutions of the singularities, generic smoothness or vanishing 
theorems would cause problems.
In order to avoid these problems, we propose to use jet schemes.

To explain the background, remember that
there are two kinds of discrepancies at a prime divisor $E$ over a variety $X$:
\begin{enumerate}
\item[$\bullet$]
 the usual log discrepancy $a(E; X)= k_E + 1$, where $k_E$ is the coefficient of the
 relative canonical divisor at the prime divisor $E$ over $X$;
 \item[$\bullet$]
 Mather-Jacobian log discrepancy $\amj(E;X)=\hk_E- j_E +1$, where $\hk_E$ is the 
 Mather discrepancy and $j_E$ is the order of the Jacobian ideal of $X$ at the prime 
 divisor $E$ over $X$.
 \end{enumerate}
 The usual log discrepancy is defined for a $\QQ$-Gorenstein variety $X$ and we say that 
 $X$ is log canonical (\resp\ canonical) at a point $x\in X$ if for every prime divisor $E$
 over $X$ with the center containing $x$ satisfies $a(E;X)\geq 0$ (\resp\  $a(E;X)\geq 1$).
 The minimal log discrepancy $\mld (x; X)$ is defined as the infimum of $a(E;X)$ for 
 every prime divisor with the center $x$. 
 This discrepancy plays an important role in the minimal model problem.
 
 The second discrepancy is defined for a reduced equidimensional scheme $X$ of finite type over 
 the base field $k$ and we say that $X$ is Mather-Jacobian (MJ, for short)-log canonical
 (\resp\ MJ-canonical ) at a point $x\in X$ if for every prime divisor $E$
 over $X$ with the center containing $x$ satisfies $\amj(E;X)\geq 0$ (\resp\  $\amj(E;X)\geq 1$).
 The MJ-minimal log discrepancy $\mldmj (x; X)$ is defined as the infimum of $\amj(E;X)$ for 
 every prime divisor with the center $x$.

When the base field $k$ is of characteristic $0$, the following natural properties hold:
\begin{enumerate}
\item[(P1)] Log canonicity, canonicity, MJ-log canonicity and MJ-canonicity are all open conditions. I.e., if $(X, x)$ is one of these singularities, then there is an open neighborhood
$U\subset X$  of $x$ such that $U$ has  singularities of the same type at every point.

\item[(P2)] Canonicity, MJ-log canonicity and MJ-canonicity are stable 
under a small deformation. (So is log canonicity if the 
total space is $\QQ$-Gorenstein.  )

\item[(P3)] The map $X\to \ZZ; x\mapsto \mldmj (x; X)$ is lower semi-continuous. 
(On the other hand, lower semi-continuity of $ \mld(x; X)$ is not yet proved in general
even in characteristic 0.)

\end{enumerate}
Resolutions of singularities played essential roles in the proofs of (P1)--(P3) for characteristic 0 case.
Therefore, none of them are proved in positive characteristic case in general.
At present, we do not have a systematic way to prove them for the usual log canonical or canonical singularities.
However, focusing on MJ-version,
we propose a potentially effective way  to prove them.
This is based on the fact 
that MJ-singularities are well described in terms of local jet schemes  at the singular point and do not need existence of a resolution of 
the singularities.

Actually  in arbitrary characteristic, MJ-minimal log discrepancy of $d$-dimensional variety $X$ at a point $x$ is represented as:
$$\mldmj(x; X)=\inf_{m\in \NN}\{(m+1)d - \dim X_m(x)\},$$
where $X_m(x)$ is the local $m$-jet scheme of $X$ at $x$ (this is proved in \cite{dd}, \cite{is} for characteristic 
0 and in \cite{ir2} for arbitrary characteristic).
Let $$s_m(X,x):=(m+1)d - \dim X_m(x).$$  

We also have the formula of $\mldmj(x;X)$ as follows (\cite{dd} and \cite{is} for characteristic $0$ case, and \cite{ir2} for arbitrary characteristic case):

Let $(X,x)\subset (A,x)$ be a closed immersion into a non-singular variety $A$ with the 
codimension $c$ and $I_X$ the ideal of $X$ in $A$, then
$$\mldmj(x;X)=\mld(x;A, I_X^c).$$
In this paper we may think that this formula is the definition of $\mldmj(x;X)$.

We pose the following conjecture for every $d\in \NN$:
\begin{conj}[$C_d$] 
\label{Cd}
Let $d$ be a positive integer.
There exists $N_d\in \NN$ depending only on $d$ such that 
for every  closed  point $x\in X$ of any $d$-dimensional variety $X$, there exists
$m\leq N_d$ satisfying either 
$$\left\{\begin{array}{l}
s_m(X,x)=\mldmj(x; X)\geq 0, \mbox{or}\\
\\
s_m(X,x)<0, \ \ \mbox{when}\  \mldmj(x; X)=-\infty \\
\end{array} \right. $$

\end{conj}

     This conjecture can be split into the following  conjectures. Let $\delta $ be an integer with 
     $ \delta \leq d$.
    Note that $\delta$ can be negative.
\begin{conj}[$C_{d,\delta}$] 
\label{split}
Let $d$ be a positive integer and $\delta$ an integer such that $\delta\leq d$.
 There exists $N_{d, \delta}\in \NN$ depending only on $d$ and $\delta$  such that 
\begin{enumerate}
\item[]
for every  closed point $x\in X$ of any $d$-dimensional variety $X$, \\
if $\mldmj(x;X)<\delta$, then there exists $m\leq N_{d,\delta}$
with the property \\
$s_m(X,x)< \delta$.
\end{enumerate}
\end{conj}

In this paper we prove the following:

\begin{thm} If Conjecture $C_d$ holds (equivalently, Conjectures $C_{d,\delta}$ holds 
for $\delta \leq d$), then the following hold:
\begin{enumerate}
\item[(PMJ1)]  Let $X$ be a $d$-dimensional variety.
If $(X, x)$ is an MJ-log canonical (\resp\ MJ-canonical) singularity, then there is an open neighborhood
$U\subset X$  of $x$ such that $U$ has MJ-log canonical (\resp\ MJ-canonical) singularity at every point of $U$.
\item[(PMJ2)] Let $\mathcal X \to \Delta$ be a surjective morphism to a smooth curve $\Delta$  with the equidimensional reduced fibers of dimension $d$. Denote the fiber of this morphism 
of a point $t\in \Delta$ by $\mathcal X_t$.  If $(X, x)=(\mathcal X_0, x)$ 
is MJ-log canonical (\resp\ MJ-canonical), then there is an open neighborhoods $\Delta'\subset
\Delta$  and $U\subset \mathcal X$ of $0$ and $x$, respectively, such that  
all fibers of $U\to \Delta'$ have MJ-log canonical (\resp\  MJ-canonical) singularities.

\item[(PMJ3)] For a $d$-dimensional variety $X$, the map 
$$\{\mbox{closed\ points\ of\ }X\} \to \ZZ;\ \  x\mapsto \mldmj (x; X)$$ is lower semi-continuous. 

\end{enumerate}
\end{thm}

  When we think of Minimal Model Problem over positive characteristic base field,
  we have to study singularities in the view point of  ``usual log discrepancy".
  Singularities with respect to MJ-log discrepancy is different from the usual one, 
  but  useful also for ``usual one".
  One reason is the fact that for singularities of locally a complete intersection 
  ``usual log discrepancy" coincides with MJ-log discrepancy.
  So the study of singularities with respect to MJ-log discrepancy is the first step to study
   singularities with respect to the usual discrepancy over positive characteristic base field.
   
  One good example is Sato-Takagi's result (\cite{sata})
  that for a quasi-projective 3-fold $X$ with canonical singularities over an algebraically closed field of characteristic
$p > 0$, a general hyperplane section of $X$ has also canonical singularities.
 They proved this by making use of a result about MJ-canonical 
  singularities  proved in \cite{ir2}.

Here, we propose a conjecture from another view point.
A similar problem for $\mld$ in characteristic zero is considered in \cite{mn}.

\begin{conj}[$D_d$] 
\label{Dd}
 For an integer $d\geq 1$, there exists $M_d\in \NN$ depending only on $d$
such that 
for any $d$-dimensional variety $X$
and a closed point $x\in X$ with a closed immersion $X\subset A$  around $x$ into a
non-singular variety $A$ of dimension $N\leq 2d$ there exists a prime divisor $E$ over $A$ with the
center at $x$ 
and $k_E \leq M_d$ such that
$$\left\{\begin{array}{l}
a(E;  A, I_X^c)= \mld (x; A, I_X^c)=\mldmj (x; X)\geq 0,  \ {\mbox or}\\
a (E; A, I_X^c)< 0\ \ {\mbox if }\ \   \mld (x; A, I_X^c)=\mldmj(x;X)= - \infty.\\
 \end{array}
 \right. $$
 Here, $c=N-d$.

\end{conj}

Note that this conjecture is not yet proved in general even for characteristic zero
(this is viewed as a special case of Conjecture 1.1 in\cite{mn}).

This conjecture claims the boundedness of necessary number of blow-ups to obtain 
a prime divisor computing the MJ-minimal log discrepancy.

In order to explain ``necessary number of blow-ups",
 we quote the basic theorem 
founded by Zariski (see, for example \cite[VI.1, 1.3]{ko}).

\begin{prop-def} 
\label{zar}
Let $X$ be an irreducible variety and $E$ a prime divisor over $X$.
Then there is a sequence of blow-ups
$$X^{(n)}\stackrel{\varphi_n}\longrightarrow X^{(n-1)}\to\cdots\to
X^{(1)}\stackrel{\varphi_1}\longrightarrow X^{(0)}=X$$
such that 
\begin{enumerate}
    \item $E$ appears on $X^{(n)}$, i.e., the 
    center of $E$ on $X^{(n)}$ is of codimension 1 and 
    $X^{(n)}$ is normal at the generic point $p_n$ of  $E$,  
  
    \item $\varphi_i(p_i)=p_{i-1}$ for $1\leq i\leq n$, and
    \item $\varphi_i$ is the blow-up with the center $\overline{\{p_{i-1}\}}$.
\end{enumerate}
The minimal such number $n$ is denoted by $\widetilde b(E)$ and the minimal number
 $i$ such that $\codim \overline{\{p_{i}\}}=1$ is denoted by $b(E)$.
 
 Note that  $\widetilde b(E)\geq b(E)$ in general 
 and the equality holds if $X$ is non-singular.
\end{prop-def}

\begin{conj}[$U_d$] 
   For every $d\in \NN$ there is an integer $B_d\in \NN$ depending only on $d$
   such that for every singularity $(X,x)$ of dimension $d$ 
   embedded into a non-singular variety $A$ with $\dim A=\emb(X,x)$,
   there is a prime divisor $E$ over $A$ computing $\mldmj(x;X)$ and satisfying 
   $\widetilde b(E)=b(E)\leq B_d$.
   Here, $\emb(X,x)$ is the embedding dimension of $X$ at $x$.
   
\end{conj}

\begin{prop}  Conjecture $C_d$, Conjecture $D_d$ and Conjecture $U_d$ are equivalent.

\end{prop}

As evidences for Conjecture $C_d$, we obtain the following:
\begin{prop}
 Conjectures $C_{d, d}$ and $C_{d, d-1}$ holds for every $d\geq 1$ and 
$ N_{d,d}=1$ and $N_{d, d-1}=5$.
\end{prop}

In Section 5 we define ``a singularity of maximal type" and show the following:

\begin{prop} Assume that Conjecture $C_{d-1}$ holds.
Then, 
Conjecture $C_d$ holds 
in the class ${\mathcal Max}_d:= \{ (X, x) \mid \dim X=d, \  (X,x) \mbox{is\ of\ 
maximal\ type} \}$.

\end{prop}
This proposition is used in the proof of the conjecture for $2$-dimensional singularities.

\begin{thm}
\label{non-deg}
Conjecture $C_d$ holds for every $d\geq 1$ in the category of non-degenerate 
\hs{s} in arbitrary characteristic.
\end{thm}

\begin{thm}
\label{1}
Conjecture $C_1$ holds and we can take $N_1=5$, $M_1=4$ and $B_1=3$  in arbitrary characteristic.

\end{thm}
\begin{thm}
\label{2} 
 If the characteristic of the base field $k$ is not $2$,
 then Conjecture $C_2$ holds and we can take $N_2\leq 41$ and $M_2\leq 58$.
 And the bound of the necessary number of blow-ups is $B_2\leq 39$.
 (Note that these numbers may not be optimal.)
\end{thm}

As a corollary of the theorem, we obtain the following statement about ``usual minimal 
log discrepancies":

\begin{cor} Assume the characteristic of the base field $k$ is not $2$.
Then, 
Conjecture $C_2$ holds  in the category of normal locally complete intersection singularities of dimension $2$  over $k$.
In particular, 
for every normal locally  complete intersection singularity $(X,x)$ of dimension $2$
there is a prime divisor $E$ over $X$  computing $\mld(x;X)(=\mldmj(x;X))$ such that 
$b(E)\leq 20$.

\end{cor}

The paper is organized as follows: In Section 2, we introduce Mather-Jacobian log discrepancies and the notion of jet schemes and then show some basic properties which
will be used in this paper.
In Section 3 we give proofs of the equivalences of the conjectures.
In Section 4 we show  that the conjecture $(C_d)$ yields the properties (PMJ1)--(PMJ3).
In Section 5 we give a proof of $(C_d)$  for non-degenerate {\hs}s and for singularities of ``maximal type" in arbitrary characteristic.
In Section 6 we give a proof of $(C_1)$  in arbitrary characteristic 
and $(C_2)$  in characteristic $\neq 2$.
For the proof of $(C_2)$ we make use of the results of Section 5.
\vskip.5truecm

{\bf Acknowledgement.} The author would like to thank Lawrence Ein for his helpful comments and encouragements in the discussions during his stay in the University of Tokyo.
The author is  grateful to Kei-ichi Watanabe for his warm encouragements while the author was working on this topic. She also thanks Kohsuke Shibata for useful discussions and J\'anos Koll\'ar
for useful comments.

\vskip.5truecm

\section{Preliminaries on Mather-Jacobian log discrepancies \\
and jet schemes}

\noindent
Throughout this paper, a variety means always an equidimensional reduced 
connected scheme of finite type over an algebraically closed field $k$.
The characteristic of $k$ is  arbitrary unless otherwise stated. 
As our discussions are local, we take a variety $X$ as an affine variety and denote its
dimension by $d$.
We use the symbol $x$ for a closed point of a scheme and $\eta$ for a not-necessarily closed point of a scheme.

Mather-Jacobian log discrepancy is defined in \cite{dd} and \cite{is} independently,
for a variety over a field of characteristic zero and is easily generalized to positive characteristic case (see, for example, \cite{ir2}).

\begin{defn} We say that $E$ is a {\it prime divisor over } $X$,
if there is a birational morphism  $Y\to X$ such that $Y$ is normal and $E$ is a
divisor on $Y$. 
A prime divisor $E$ over $X$ is called an {\it exceptional prime divisor over } $X$,
if  the morphism $Y\to X$ is not isomorphic at the generic point of $E$.
\end{defn}

\begin{defn}
Let $E$ be a prime divisor over an arbitrary variety $X$.
The {\it  Mather-Jacobian (MJ, for short) log discrepancy} of $X$ at $E$ is defined by
$$\amj(E;X)=\hk_E- j_E +1,$$  
where $\hk_E$ is the order of vanishing of the relative Jacobian ideal $\j_\varphi = \Fitt^0(
\Omega_{Y/X})$ at $E$  for a partial resolution $\varphi:Y \to X$ on which $E$ appears.
The other term $j_E$ is the order of vanishing  of the Jacobian ideal $\j_X$ of $X$ at  $E$. 
\end{defn}
Here, we note that MJ-log discrepancy is defined on every variety,
while the usual log discrepancy $a(E; X)$ is defined only for normal $\QQ$-Gorenstein variety.
We also note that for locally a complete intersection $X$ we have the coincidence
$\amj(E;X) =a(E;X)$.
Detailed discussions for MJ-log discrepancies  can be seen in \cite{ei} and  \cite{ir2}.

\begin{defn} For a (not necessarily closed) point $\eta\in X$  and for a proper closed subset $W$, we define the {\it minimal MJ-log discrepancy } at $\eta$ as follows:
\begin{enumerate}
\item When $\dim X\geq 2$,
$$\mldmj (\eta; X)=\inf\{ \amj(E; X) \mid  E :\ {\operatorname { prime\ divisor \ with \
center}}\ \overline{\{\eta\}} \} .$$ 
$$\mldmj (W;X)=\inf\{ \amj(E; X) \mid  E :\ {\operatorname { prime\ divisor \ with \
center\ in}}\ W \} .$$ 
\item When $\dim X=1$, define $\mldmj (\eta; X)$ and $\mldmj (W;X)$ by the same definitions as above if the right hand sides
of the above definitions 
are non-negative and otherwise define $\mldmj (\eta; X)= -\infty$ and $\mldmj (W;X)=
-\infty$, respectively.

Here, we emphasize that ``with center $\overline{\{\eta\}}$"  means that 
the center coincides with $\overline{\{\eta\}}$, and different from 
``with center `in' $\overline{\{\eta\}}$".

\end{enumerate}
\end{defn}

The following is well known  (see, for example, \cite{ir2}):

\begin{prop} The inequality $$\mldmj(\eta; X)\leq \codim(\overline{\{\eta\}}, X)$$ holds
for a point $\eta\in X$, and the equality holds if and only if $(X,\eta)$ is non-singular.
In particular, if $x\in X$ is a closed point, the inequlaity
$$\mldmj(x; X)\leq \dim X$$ holds, and the equality holds if and only if $(X,x)$ 
is non-singular.

\end{prop}

\begin{defn}
 A variety $X$ is called {\it MJ-canonical (resp. MJ-log canonical)}
 at a (not necessarily closed) point $\eta\in X$,
if $$\amj(E; X)\geq 1\ \ \ \  (\rm{resp. }\ \geq 0)$$
for every exceptional prime divisor $E$ over $X$ whose center on $X$ contains $\eta$.

\end{defn}

\begin{prop}
\begin{enumerate}
\item A variety $X$ is MJ-log canonical at a  point $\eta$ if and only if 
$\mldmj(\eta; X)\geq 0$.
\item If a variety $X $ is MJ-canonical at $\eta$, then $\mldmj(\eta; X)\geq 1$ holds by definition. But the converse does not hold in general.
\end{enumerate}
\end{prop}

The MJ-version of the singularities has a  good description in terms of jet schemes.
Here, we introduce jet schemes.
The precise descriptions about jet schemes and the arc spaces are found,
for example in \cite{EM}, \cite{icr}.

\begin{defn}
 Let \( X \) be a variety and  $K\supset k$ a field extension.
For  \( m\in \ZZ_{\geq 0} \) a \( k \)-morphism \( \spec K[t]/(t^{m+1})\to X \) is called an  {\it{\( m \)-jet}} of \( X \) and 
 \( k \)-morphism \( \spec K[[t]]\to X \) is called an {\it {arc}} of \( X \).
\end{defn}


Let 
 \( X_{m} \) be the {\it space of \( m \)-jets} or  the $m$- {\it{jet scheme}}
   of  \( X \). 
   There exists the projective limit $$X_\infty:=\lim_{\overleftarrow {m}} X_m$$
   and it is called the {\it space of arcs} or the {\it arc space} of $X$.

\begin{defn}
 Denote the canonical truncation morphisms induced from $k[[t]]\to k[t]/(t^{m+1})$ 
 and $k[t]/(t^{m+1})\to k$  by
 $\psi_m: X_\infty\to X_m$ and $\pi_m: X_m\to X$, respectively.
 In particular we denote the  morphism  $\psi_0=\pi_\infty : X_\infty \to X$ by $\pi$.
 We also denote the canonical truncation morphism $X_{m'} \to X_m$ $(m'> m)$ induced 
 from $k[t]/(t^{m'+1})\to k[t]/(t^{m+1})$ by
 $\psi_{m', m}$.
 To specify the space $X$, we  sometimes write $\psi^X_{m', m}$.
 
 For a point $\eta\in X$, the fiber scheme $\pi_m^{-1}(\eta)$ is denoted by $X_m(\eta)$
 and the dimension $\dim X_m(\eta)$ is defined by $\dim \overline{\pi_m^{-1}(\eta)}$.
 Note that it is different from $\dim \pi_m^{-1}(\overline{\{\eta\}})$.

\end{defn}

\begin{defn}
For an arc $\gamma\in X_\infty$, 
the {\it order of a coherent ideal $\a\subset \o_X$ measured by $\gamma$}
is defined as follows:
let $\gamma^*: \o_{X, \gamma(0)} \to K[[t]]$
be the corresponding
ring homomorphism of $\gamma$,
where $0\in \spec  K[[t]]$ is the closed point.
Then, we define
$$\ord_\gamma(\a)=\sup \{ r\in \ZZ_{\geq 0}\mid \gamma^*(\a)\subset (t^r)\},$$
We define the subsets ``contact loci" in the arc space as follows:
  $$\cont^m(\a)=\{\gamma \in X_\infty \mid \ord_\gamma(\a)=m\}$$
In a similar manner, we define
 $$\cont^{\geq m}(\a)=\{\gamma \in X_\infty \mid \ord_\gamma(\a)\geq m\}$$
By this definition, we can see that
$$\cont^{\geq m}(\a)=\psi_{m-1}^{-1}(Z(\a)_{m-1}),$$
where  $Z(\a)$ is the closed subscheme defined by the ideal $\a$ in $X$.

\end{defn}

The following fact was already used in \cite{ei}, \cite{ir}, \cite{ir2} in the proofs  of some
statements.
Here, we give the proof for the reader's convenience
as we will use it several times in this paper.
About the basic terminologies appear in the proof, we refer to \cite{icr}.

\begin{prop}
\label{truncation}
  Denote $\AA_k^N=\spec k[x_1,\ldots, x_N]$ just by $\AA^N$
  in order to avoid confusion with the ``$k$-jet scheme".
  Let $X\subset \AA^N$ be defined by an ideal $I\subset k[x_1,\ldots, x_N]$
  and contain the origin $0$ and  let $Z$ be defined
by the  ideal generated by the $m$-truncations of the elements of $I$.
Here the $m$-truncation of a polynomial $f=\sum _{i=0}^r f_i$ ($f_i$ is homogeneous
of degree $i$) is $\sum _{i=0}^m f_i$.
Then for every $j\leq m$, the local $j$-jet schemes coincide
$$X_j(0)=Z_j(0).$$
In particular, if for every element $f\in I_X$ has order greater than $m$, then for every 
$j\leq m$,
$$X_j(0)=\AA_j^N(0).$$
\end{prop}

\begin{proof} The $m$-jet scheme $\AA^N_m$ of the affine space $\AA^N$ is 
given as 
$$\AA^N_m=\spec   k\left[x_i^{(j)} \mid i=1,\ldots, N,\ j=0,1,\ldots, m\right]\simeq\AA^{N(m+1)}.$$
Here, the closed point $P\in \AA^N_m$ with the coordinates $(a_i^{(j)})_{i=1,\ldots, N,\ j=0,1,\ldots, m}$ 
corresponds to the $m$-jet $\spec  k[t]/(t^{m+1}) \to \AA^N$ whose associated 
$k$-algebra homomorphism is
$$k[x_1,\ldots, x_N] \to k[t]/(t^{m+1}), \ \ x_i\mapsto \sum_{j=0}^m {a_i}^{(j)} t^j.$$
Let $I\subset k[x_1,\ldots, x_N] $ be the defining ideal of $X$ in $\AA^N$.
Then the defining ideal $$I_m\subset \ k\left[x_i^{(j)} \mid i=1,\ldots, N,\ j=0,1,\ldots, m\right]$$
of $X_m$ in $\AA^N_m$ is generated by $$\left\{ f^{(j)} \mid  j=0,1,\ldots, m,\  f\in I\right\},$$
where $f^{(j)}\in k\left[x_i^{(l)} \mid i=1,\ldots, N,\ l=0,1,\ldots, m\right]$ is defined 
from $f\in I$ as follows:
\begin{equation}
\label{derivation}
f\left(\sum_{j=0}^m x_1^{(j)}t^j,\ldots, \sum_{j=0}^m x_N^{(j)}t^j\right)\equiv \sum_{j=0}^m
 f^{(j)}t^j\  (\mod t^{m+1}).
 \end{equation}
 Here, we define the weight of the variable $x_i^{(j)}$ by $j$ and the weight of a
 monomial by the sum of the weights of variables appearing in the monomial.
 Then the polynomial $f^{(j)}$ is homogeneous of weight $j$.
 
 We note that $f^{(j)}$ may have  another homogeneity. 
 If $f$ is homogeneous with respect to the usual degree,
 then $f^{(j)}$ is also homogeneous  with respect to the usual degree.
 For a general $f$, let $f=f_0+f_1+\cdots + f_r$ be the homogeneous decomposition with respect to 
 the degree and $f_i$ is the homogeneous part of degree $i$.
 Then 
 \begin{equation}
 \label{homogeneous}
 f^{(j)}=f_0^{(j)} +f_2^{(j)}+\cdots+ f_r^{(j)},
 \end{equation}
 where $f_i^{(j)}$ is defined in the same way as in (\ref{derivation}) for the homogeneous polynomial $f_i$.
 
 Now we consider the fibers $\AA^N_m(0)$ and $X_m(0)$ by $\pi_m^{\AA^N}$
 and $\pi_m^{X}$, respectively.
 These are defined by $x_i^{(0)}=0$ for all $i=1,\ldots, N$, in $\AA^N_m$ 
 and in $X_m$, respectively.
 Therefore, $X_m(0)$ is defined by $\{\overline{f^{(j)}} \mid j=0,1,\ldots, m, \ 
 f\in I\}$ in $$\AA^N_m(0)=\spec  k\left[x_i^{(j)} \mid i=1,\ldots, N,\ j=1,\ldots, m\right]
 \simeq\AA^{Nm},$$ 
 where $\overline{f^{(j)}}$ is the polynomial substituted $x_i^{(0)}=0$ ($i=1,\ldots, N$)
 into $f^{(j)}$.
 Under the notation in (\ref{homogeneous}), we note that
 \begin{equation}
 \label{vanish}
 \overline{f_i^{(j)}}=0\ \ \mbox{ if}\ \ i>j.
 \end{equation} 
 This is because the monomials of $f_i^{(j)}$ are homogeneous of degree $i$  of the form $x_{s_1}^{(j_1)}\cdot x_{s_2}^{(j_2)}\cdots
 \cdot x_{s_i}^{(j_i)}$ with \\ $j_1+j_2+\cdots+j_i=j$.
 Here, by $i>j$, there is some $l, (1\leq l\leq i)$ such that $j_l=0$.
 This yields that every monomial of  $f_i^{(j)}$ contains $x_{s_l}^{(0)}$ as a factor 
 for some $s_l$, therefore
 by the substitution $x_{s_l}^{(0)}=0$, the monomial becomes 0.
Thus it follows (\ref{vanish}).

Since $X_m(0)$ is defined by $\left\{ \overline{f_i^{(j)}} \mid j\leq m,\ f\in I \right\}$,
the vanishing (\ref{vanish}) implies that the homogeneous parts $f_i$ $(i>m)$ do not affect the
defining ideal of $X_m(0)$.
Therefore  $X_m(0)=Z_m(0)$. 
\end{proof}

\begin{thm}[\cite{dd}, \cite{is}, \cite{ir2}]
\label{IA} 
Let $X$ be a variety over an algebraically
closed field $k$ of an arbitrary characteristic and $A$ a smooth variety containing $X$ as a closed
subscheme of codimension $c$.
 Denote the ideal of $X$ in $A$ by $I_X$.
Then, for a proper closed subset  $W$ of  $X$, we have 
\begin{equation}
\label{iaclosed}
\mldmj(W; X)=\mld(W;A, I_X^c).
\end{equation}
For a point $\eta\in X$, we have 
\begin{equation}
\label{iapoint}
\mldmj(\eta; X)=\mld(\eta;A, I_X^c).
\end{equation}
\end{thm}

By the above theorem we may think that  the right hand sides of (\ref{iaclosed})
and (\ref{iapoint}) are the definitions of $\mldmj(W; X,)$ and $\mldmj(\eta; X)$, respectively.

\begin{prop}[\cite{dd}, \cite{is}, \cite{ir2}]
\label{formula1}
Let $X$ be a variety of dimension $d$  embedded into a non-singular variety $A$ with
codimension $c$ and let $\eta\in X$ be a point, then we have 
$$\mldmj(\eta; X)= \inf_{m\in\NN}\left\{ \codim\left( \cont^{\geq m+1}(I_X)\cap (\pi^A)^{-1}(\eta), A_\infty\right) - c\cdot (m+1)\right\}$$
$$ = \inf_{m\in\NN}\left\{ (m+1)d - \dim X_m(\eta) \right\},\ \ \ \ \ \ \ \ \ \ \ \ \ \ \ \ \ \ \ $$
where $\dim X_m(\eta)=\dim \overline{ X_m(\eta)}$.
Note that it does not coincide with $\dim { X_m(\overline{\{\eta\}})}$
in general.
\end{prop}

\begin{defn}
Under the assumption of the previous proposition, 
for a point $\eta\in X$ we define the function $s_m(X,\eta)$ in $m$ 
as follows:
$$s_m(X,\eta)=(m+1)d - \dim X_m(\eta).$$

\end{defn}

\begin{defn}
For a variety $X$, fix a closed immersion $X\subset A$ into a non-singular variety $A$.
\begin{enumerate}
\item
We say that a prime divisor $E$ over $A$ with the center $\overline{\{\eta\}}$ computes $\mldmj(\eta; X )$ if either
$$a(E; A,  I_X^c)=\mld(x;A, I_X^c) =\mldmj(\eta; X)\geq 0,\ \ \mbox{or}$$
$$a(E; A,  I_X^c)< 0,  \ \mbox{when}\ \ \mldmj(\eta; X)=-\infty.$$
\item
We say that $s_m(X,\eta)$ computes $\mldmj(\eta; X)$ if either
$$s_m(X,\eta)=\mldmj(\eta; X)\geq 0, \ \ \mbox{or}$$
$$s_m(X,\eta)<0,  \ \mbox{when}\ \ \mldmj(\eta; X)=-\infty.$$

\end{enumerate}

\end{defn}

\vskip.5truecm

\section{The conjectures and their relations}
\noindent
In ths section we prove the equivalence of Conjecture $C_d$ with the other conjectures. 

\begin{defn}
Let $X$ be a $d$-dimensional variety and $\eta\in X$  a point. 
Let $A$ be a non-singular variety containing $X$ as a closed subscheme.
 We define the following invariants:
$$\nu(X,\eta):= \min\{r \mid s_{r-1}(X,\eta)\ \mbox{computes} \ \ \mldmj(\eta,X)\}.$$
   
\end{defn}

Then, in terms of $\nu$, Conjecture $C_d$  (Conjecture \ref{Cd})  is represented as follows:

\begin{conj}[$C_d$] 
For an integer $d\geq 1$, there exists $N_d>0$ depending only on $d$
such that the bound $\nu(X,x)-1 \leq N_d$ holds
for any $d$-dimensional variety $X$ and a closed point $x\in X$.
\end{conj}

First we pose another conjecture which seems stronger than ($C_d$).
Actually Conjecture $C_d$ is the statement for closed points,
while the following conjecture is for any points in $d$-dimensional varieties.

\begin{conj}[$\widetilde C_d$] 
\label{tCd}
For an integer $d\geq 1$, there exists $N_d>0$ depending only on $d$
such that the bound $\nu(X,\eta)-1 \leq N_d$ holds
for any $d$-dimensional variety $X$ and a  point $\eta\in X$.
\end{conj}

\begin{prop}
\label{closed&nonclosed}
 For an integer $d\geq 1$, Conjecture $C_d$ and Conjecture $\widetilde C_d$
 are equivalent.
\end{prop}
\begin{proof}
The implication $\widetilde C_d \Rightarrow C_d$ is obvious.
To show the converse implication, take any non-closed point $\eta\in X$.
Let $N_d$ be as in Conjecture $C_d$.
Let $\nu=\nu(X, \eta)$, and we will show that $\nu-1 \leq N_d$.
Note that for every $m\in \NN$, there exists an open dense subset $U_m\subset
\overline{\{\eta\}}$ such that
$$\dim X_m(\eta)=\dim X_m(x)+\dim\overline{\{\eta\}}$$
holds for every closed point $x\in U_m$.
Then for  every closed point  $x\in  U_m$ we have
$$
s_m(X, \eta)=d(m+1)-\dim X_m(\eta)
=s_m(X, x)-\dim \overline{\{\eta\}}.$$

Then,   for a closed point $$x\in U_\nu\cap \left(\bigcap_{n\leq N_d}U_n\right),$$ 
we obtain

$$\mldmj(\eta, X)=s_{\nu-1} (X,\eta)=s_{\nu-1}(X, x)-\dim \overline{\{\eta\}}$$
\begin{equation}
\label{nonclosed}
\geq s_{\nu(X,x)-1}(X, x)-\dim \overline{\{\eta\}}=s_{\nu(X,x)-1 }(X,\eta)\geq \mldmj(\eta, X).
\end{equation}

Therefore all inequalities in (\ref{nonclosed}) become equalities.
By the minimality of $\nu=\nu(X,\eta)$, we obtain that $\nu-1\leq \nu(X,x)-1\leq N_d$.
\end{proof}

By this proposition, we reduce the problem in $(\widetilde C_d)$ into the problem on closed points.
Henceforth, we will consider the conjectures only for closed points.

\begin{lem}
\label{nu=mu}  
Let $X$ be a variety and $x\in X$ a closed point.
We assume that $X$ is embedded into a non-singular variety $A$ with codimension $c$.
Then, we have  $$\nu(X,x)=      \min \left\{ m\in \NN \left| 
\begin{array}{l}  k_E+1-cm=\mldmj(x;X)\\
\mbox{for\ a prime\ divisor}\ E\  \mbox{over \ $A$\ computing}\ \mldmj(x; X)\\
 \end{array} \right\},\right.$$  
     where in case $\mldmj(x;X)=-\infty$, the condition $k_E+1-cm=\mldmj(x;X)$ means
     that $k_E+1-cm<0$.
     
      In case $\mldmj(x;X)\geq 0$, the statement gives
    $$\nu(X,x)=\min \left\{\val_E(I_X) \mid \mbox{for}\ E\  \mbox{computing}\ \mldmj(x; X)\right\}.$$
\end{lem}

\begin{proof}
The last statement is obvious, once we prove the main statement.
Because the condition $k_E+1-cm=\mldmj(x;X)=\mld(x; A, I_X^c)$ implies $m=\val_E(I_X)$ under the assumption 
$\mldmj(x;X)\geq 0$.
 
 For the proof of the main statement, we define
 $$\mu(X, x):=  \min \left\{ m\in \NN \left| 
\begin{array}{l}  k_E+1-cm=\mldmj(x;X)\\
\mbox{for\ a prime\ divisor}\ E\  \mbox{over \ $A$\ computing}\ \mldmj(x; X)\\
 \end{array} \right\},\right.$$  
 and will prove that $\nu(X, x)=\mu(X, x)$.

For simplicity of notation on proving the main statement of the lemma, we denote $\mu(X,x)$ and $\nu(X,x)$ by just $\mu$ and $\nu$,
respectively.

{\bf Case 1.}  $\delta:=\mldmj(x;X)\geq 0$.

In this case, $$\delta=s_{\nu-1}(X,x)=\codim (\cont^{ \geq \nu}(I_X)\cap (\pi^A)^{-1}(x), A_\infty)
          -c \cdot \nu$$
   Let $C$ be an irreducible component of  $\cont^{\geq \nu}(I_X)\cap (\pi^A)^{-1}(x)$
   that gives the codimension. 
   Then,  ${C}$ is a maximal divisorial set $C_A(q\cdot \val_E)$ 
   for some $q\in \NN$ and a prime divisor $E$ over $A$ with the center $x$  (see, for example,
   \cite[Corollary 3.16]{ir2}).
   As $C\subset \cont ^{\geq\nu}(I_X)$, it follows that $$q\cdot\val_E(I_X)\geq\nu.$$
   On the other hand, by \cite[Lemma 2.7]{zhu} (see also \cite[Theorem 3.13]{ir2}) it follows 
   $$\codim (C, A_\infty)=\codim (C_A(q\cdot\val_E), A_\infty)=q(k_E+1).$$ 
   Therefore, we obtain
   $$\delta=s_{\nu-1}(X,x)=q(k_E+1)- c\nu\geq q(k_E+1-c\cdot\val_E(I_X))
   \geq k_E+1-c\cdot\val_E(I_X)\geq \delta.$$
 Then, all inequalities become equalities.
Therefore,  the last equality shows that $E$ computes $\mldmj(x;X)$ and $\val_E(I_X)\leq\nu$.
(More precisely, $\val_E(I_X)=\nu$ if $\delta>0$ and $\val_E(I_X)=\nu/q\leq \nu$ if $\delta=0$.)
 Hence, by the minimality of $\mu$, we obtain 
  $$\mu\leq \nu.$$
  
  Next we prove the converse inequality $\mu\geq \nu.$
  Let $E$ be a prime divisor over $A$ with the center  $x$ computing $\mldmj(x;X)=\delta$ and $\val_E(I_X)=\mu$.
  As $C_A(\val_E)\subset \cont^{\geq \mu}(I_X)\cap(\pi^A)^{-1}(x)$,
  $$\codim( \cont^{\geq \mu}(I_X)\cap(\pi^A)^{-1}(x), A_\infty)\leq 
  \codim (C_A(\val_E), A_\infty)=k_E+1.$$
  Therefore, we obtain
  $$s_{\mu-1}(X,x)=\codim ( \cont^{\geq \mu}(I_X)\cap(\pi^A)^{-1}(x), A_\infty)-c\cdot\mu$$
  $$\leq k_E+1-c\cdot\val_E(I_X)=\delta.$$
  Hence the equality holds, which yields $\nu\leq \mu$ by the minimality of $\nu$.
  
  {\bf Case 2.} 
   $\mldmj(x; X)=-\infty$.
  Then, by the same argument as  in Case 1 we can write
  $$s_{\nu-1}(X,x)=\codim\left( \cont^{\geq \nu}(I_X)\cap \pi^{-1}(x), A_\infty\right )-c\nu$$
$$\ \ \ \ \ \ \ \ \   =\codim C(q\cdot\val_E)-c\nu =q(k_E+1)-c\nu <0,$$
    for some $q\in \NN$ and a prime divisor $E$ over $A$ with the center  $x$ such that   
     $$q\cdot~\val_E(I_X)\geq\nu.$$ 
    Therefore, $$q(k_E+1-c\cdot\val_E(I_X))\leq q(k_E+1)-c\nu<0, $$
which implies that     $E$ computes $\mldmj(x; X)=-\infty$.
  Then, by the minimality of $\mu$, it follows that $\mu\leq \lceil\frac{\nu}{q}\rceil\leq \nu$.

  Now, to show the converse inequality, 
  take a prime divisor $E$ over $A$ computing $\mldmj(x;X)=-\infty$ and satisfies
 $k_E+1-c\cdot\mu <0$.
By the minimality of $\mu$ we have $\mu\leq \val_E(I_X)$,
therefore $C(\val_E)\subset \cont^{\geq \mu}(I_X)\cap\pi^{-1}(x)$.
  By the same argument as the corresponding part in Case 1, we observe 
 $$s_{\mu-1}(X,x)\leq k_E+1-c\cdot \mu<0.$$
 Then, by the minimality of $\nu$, we obtain $\nu\leq \mu$.
\end{proof} 

As  $\nu$ does not depend on the choice of a closed immersion $X\subset A$ into a non-singular variety $A$, we obtain the following:

\begin{cor} Let $X$ be embedded into a non-singular variety $A$ with codimension $c$.
The invariant $$\mu(X,x)=      \min \left\{ m\in \NN \left| 
\begin{array}{l}  k_E+1-cm=\mldmj(x;X)\\
\mbox{for\ a prime\ divisor}\ E\  \mbox{over \ $A$\ computing}\ \mldmj(x; X)\\
 \end{array} \right\},\right.$$  is independent of the choice of a closed immersion into a 
non-singular variety.

\end{cor}

Now we can interpret the conjecture into a more birational theoretic  Conjecture~$D_d$ (Conjecture \ref{Dd}) 
  as follows:

\begin{prop}\label{C=D}
 Conjecture ${C_d}$ and Conjecture $D_d$ are equivalent.
\end{prop}

\begin{proof}
        First we show that (${C_d})$ implies ($D_d$).
        
        Let $x$ be a closed point of a $d$-dimensional variety $X$ 
        embedded into a non-singular
        variety $A$ of dimension $N\leq 2d$.
         Let a prime divisor $E$ over $A$ compute $\mldmj(x; X)$
        and satisfy 
        $$k_E-c\cdot\mu(X,x)+1=\mldmj(x; X)\leq d,$$ 
        where $c=\codim (X, A)\leq d$.
        Here, in case $\mldmj(x;X)=-\infty$,  the above equality implies
        $k_E-c\cdot\mu(X,x)+1<0$, as in Lemma \ref{nu=mu}.
        Then, by the assumpton ($C_d$) and Lemma \ref{nu=mu}, we have
        $$k_E\leq d + d({N_d}+1)-1,$$
        therefore Conjecture $D_d$  (Conjecture \ref{Dd}) holds and 
        we can take $$M_d\leq d(N_d+2)-1.$$

\vskip.5truecm
Next we show that (${D_d}$) implies ($C_d$). 
            First we consider the case that there is a closed immersion $X\subset A$ into 
            a non-singular variety $A$ of dimension $N\leq 2d$ around $x$,
            so that we can apply the statement of ($D_d$).
            
            When $\mldmj(x;X)\geq 0$, by the condition ($D_d$), 
            there is a prime divisor $E$ over $A$ such that 
            $$ k_E\leq M_d  \ \ \mbox{and}$$
            $$ k_E-c\cdot\val_E I_X +1 =\mldmj(x;X)\geq 0.$$
            This yields the following bound:
            $$\mu(X,x)=\val_E I_X\leq \frac{k_E+1}{c}\leq \frac{M_d+1}{c}\leq M_d+1.$$
            When $\mldmj(x;X)=-\infty$, by the condition ($D_d$), 
            there is a prime divisor $E$ over $A$ computing $\mldmj(x;X)$ such that 
            $$ k_E\leq M_d \ \ \mbox{and}$$ 
            $$k_E-c\cdot \mu(X,x)+1<0.$$ 
            By the definition,  $\mu(X,x)$ is the minimal integer satisfying the
            last inequality, which yields that 
            $$\mu(X,x)\leq \frac{k_E+2}{c}\leq M_d+2.$$

         Next we assume that $\emb(X,x)>2d$, then we have $\dim X_{1}(0) > 2d $ which  
         implies $s_1(X,x)<0$.
         In this case, automatically $\mldmj(x;X)=-\infty$
         and $$\mu(X,x)=\nu(X,x)=2.$$
         Hence, for all cases, we obtain $(C_d)$ and we can take 
         $$N_d\leq  M_d+1.$$
\end{proof}

   These conjectures are also equivalent to the following conjecture implying
    the boundedness of the number of blow-ups 
   to obtain a prime divisor computing the MJ-minimal log discrepancy.
   
\begin{conj}[$U_d$] 
   For every $d\in \NN$ there is an integer $B_d\in \NN$ depending only on $d$
   such that for every singularity $(X,x)$ of dimension $d$ 
   embedded into a non-singular variety $A$ with $\dim A=\emb(X,x)$,
   there is a prime divisor $E$ over $A$ computing $\mldmj(x;X)$ and
   $\widetilde b(E)=b(E)\leq B_d$.
   
\end{conj}   

\begin{prop}
\label{C=U}
 Conjecture $U_d$ is equivalent to Conjecture $C_d$ and $D_d$.
\end{prop}

\begin{proof} First we will show the implication $(D_d) \Rightarrow (U_d)$.
  
    When $\emb(X,x)=\dim A=N> 2d$, then the exceptional divisor $E_1$ obtained by
     the  blow-up of $A$ at a point $x$ computes $\mldmj(x;X)=-\infty$.
     Indeed, 
     $$a(E_1; A, I_X^c)=k_{E_1}+1-c\cdot\val_EI_X\leq N-2(N-d)<0.$$
     
     So we may assume that $N\leq 2d$.
     By the assumption ($D_d$), 
     there is a  prime divisor $E$ over $A$ computing $\mldmj(x; X)$ such that $k_E\leq M_d$.
      Let $c_i$ be the codimension in $A$ of the center  of the $i$-th blow-up 
      ($1\leq i\leq \tilde b(E))$.
     Then $c_1=\emb(X,x)\geq d+1$ and $c_i\geq 2$.
     Therefore it follows
     $$M_d\geq k_E \geq \sum_{i=1}^{\tilde b(E)}( c_i-1)\geq  d+(\tilde b(E)-1)=\tilde b(E)+d-1.$$
     Hence, we obtain $\tilde b(E) \leq M_d-d+1$, which yields the positive answer to Conjecture $U_d$ 
     and 
     $$B_d\leq M_d-d+1.$$
     
     Next we prove the converse $(U_d) \Rightarrow (D_d)$.
       Let $(X,x)$ be any $d$-dimensional singularity embedded into a 
     non-singular variety $A$ of dimension $\leq 2d$.
    Let $E$ be a prime divisor over $A$ computing $\mldmj(x;X)$ such that 
    $\tilde b(E)\leq B_d$.  
    As $\ord _{E_i}K_{A_i/A_{i-1}}\leq 2d-1$, where the left hand side is
    the coefficient of the divisor $K_{A_i/A_{i-1}}$ at $E_i$,  we obtain 
    $$k_E\leq 2^{\tilde b(E)-1}(2d-1).$$
    This  implies the Conjecture $D_d$ and $M_d\leq 2^{B_d-1}(2d-1)$.

\end{proof}

When one tries to prove Conjecture $C_d$, it may be useful to split it into 
small conjectures $C_{d, \delta} $ for every integer $\delta\leq d$ as follows:

\begin{conj}[$C_{d, \delta}$]
\label{split}
Let $d$ and $\delta$ be as above.
 There exists $N_{d, \delta}\in \NN$ depending only on $d$ and $\delta$  such that 

for every  closed point $x\in X$ of any $d$-dimensional variety $X$ with
$\mldmj(x;X)<\delta$, there exists $m\leq N_{d,\delta}$
with the property 
$s_m(X,x)< \delta$.
\end{conj}

Actually we have the following equivalence:
\begin{prop}
\label{3equi}
The following are equivalent:
\begin{enumerate}
\item Conjecture $C_d$,
\item Conjecture $C_{d,\delta}$  for every integer $\delta$ such that 
$0\leq \delta\leq d$,
\item Conjecture $C_{d,\delta}$  for every integer $\delta\leq d$.

\end{enumerate}

\end{prop}

\begin{proof}
  For the proof (1) $\Rightarrow$ (2), we can take $N_{d, \delta}:=N_d$
  for every $\delta$ with $0\leq \delta\leq d$.
  For the converse (1) $\Leftarrow$ (2), we can take $N_d:=\max _\delta N_{d,\delta}$.
  The implication (3) $\Rightarrow$ (2) is obvious.
  For the proof of (2) $\Rightarrow$ (3), it is sufficient to show that 
  for every $\delta=-i<0$, we can take $N_{d,-i}:= (i+1)(N_{d,0}+1)-1$.
  Indeed, by the assumption, there is $m\leq N_{d,0}$ such that
  $$s_{m}(X,x)=\codim\left( \cont^{\geq m+1}(I_X)\cap (\pi^A)^{-1}(\eta), A_\infty\right) - c\cdot (m+1)\leq -1,$$
  which implies that there is an integer $q$ and a prime divisor $E$ over $A$ with the center $x$
  such that $q\cdot\val_EI_X\geq m+1$ and
  
  $$q(k_E+1)-c(m+1)=s_{m}(X,x)\leq-1.$$
  Then, since the maximal divisorial set
  $C_A((i+1)q\cdot\val_E)$ satisfies the following
   $$C_A((i+1)q\cdot\val_E)\subset \cont^{\geq(i+1)(m+1)}I_X\cap (\pi^A)^{-1}(x),$$ 
   it follows
  $$s_{(i+1)(m+1)-1}(X,x)\leq (i+1)q(k_E+1)-(i+1)c(m+1)=(i+1) s_{m}(X,x)  \leq -(i+1),$$
  therefore we can take $N_{d,-i}:= (i+1)(N_{d,0}+1)-1$.

\end{proof}

\section{Applications of the conjecture}
\noindent
In this section we prove the properties (P{MJ}1)-(P{MJ}3) under the assumption that
Conjecture $ C_d $ holds.
First we prove (P{MJ}3). 
The following is a relative version of (PMJ3) and the absolute version follows immdiately as a special case.
\begin{prop} [Lower semi-continuity]
\label{lower}
Assume that Conjecture $ C_d $ holds for an integer $d\geq 1$. 
Let $\rho: {\mathcal X} \to Y$ be a surjective morphism of varieties with the $d$-dimensional
varieties as fibers. 
Let us denote the fiber   $\rho^{-1}(y)$ of $y\in Y$  by ${\mathcal X}_y$. 
Consider the map ${\mathcal X}\to \ZZ$ associating a closed point $x\in {\mathcal X}$ to $\mldmj (x; {\mathcal X}_{\rho(x)})$.
Then the map is lower semi-continuous, i.e., if 
$$\mldmj (x; {\mathcal X}_{\rho(x)}) = \delta,$$
then there is an open neighborhood ${\mathcal U}\subset {\mathcal X}$ of $x$ such that 
for all closed point $x'\in {\mathcal U}$, 
$$\mldmj (x' ; {\mathcal X}_{\rho(x')}) \geq \delta.$$

In particular, in case $Y=\spec  k$, then
  the map $X\to \ZZ; x\mapsto \mldmj (x; X)$ is lower semi-continuous. 

\end{prop}

\begin{proof} 
Only in this proof we denote the $m$-jet scheme of a variety ${\mathcal X}$ by 
${\mathcal L}_m({\mathcal X}) $ and relative $m$-jet scheme of ${\mathcal X}$
with respect to $\rho$ by ${\mathcal L}_m({\mathcal X}/Y)$ for the convenience
to distinguish them from the fibers ${\mathcal X}_t$  of $\rho$. 
The definition/construction of the relative $m$-jet scheme is given in \cite[Proof of Theorem 4.9]{ei} and also \cite[Proof of Proposition 2.3]{m}.
The relative $m$-jet scheme ${\mathcal L}_m({\mathcal X}/Y)$ is a closed 
subscheme of ${\mathcal L}_m({\mathcal X}) $ such that $\pi_m^{-1}({\mathcal X}_y)={\mathcal L}_m({\mathcal X}_y) $,
where $\pi_m: {\mathcal L}_m({\mathcal X}/Y)\to {\mathcal X}$ is the canonical 
truncation morphism.

For every $m\in \NN$, 
$${\mathcal X}\to \ZZ, \ x\mapsto  \dim\pi_m^{-1}(x)=\dim { \left(\pi_m^{\mathcal X_{\rho(x)}}\right)}^{-1}(x)$$ is upper semi-continuous (see, for example \cite[4.11]{ei}).
Here, $\pi_m^{\mathcal X_{\rho(x)}}: {\mathcal L}_m({\mathcal X_{\rho(x)}})\to 
{\mathcal X_{\rho(x)}}$ is the canonical truncation morphism and also the restriction 
of $\pi_m$ on ${\mathcal L}_m({\mathcal X_{\rho(x)}})$.
Therefore $$s_m({\mathcal X}_{\rho(x)},x)=(m+1)d-\dim \pi_m^{-1}(x)=(m+1)d-\dim { \left(\pi_m^{\mathcal X_{\rho(x)}}\right)}^{-1}(x)$$ is lower semi-continuous for all $m\in \NN$.
The Conjecture $C_d $ implies that 
$$\mldmj (x; {\mathcal X}_{\rho(x)})=\min\{s_m({\mathcal X}_{\rho(x)},x) \mid m\leq {N_d}\}.$$
Hence, $\mldmj (x; {\mathcal X}_{\rho(x)})$ is lower semi-continuous.
\end{proof}

\begin{rem}
\label{stratification}
 Assume that Conjecture $ C_d $ holds. 
For $\delta=-\infty, 0, 1, \ldots d$, let $X(\delta)$ be the locally closed subset formed by the closed points $x\in X$
such that $\mldmj(x; X)=\delta$.
Then $\{X(\delta)\}_\delta$ is a finite stratification of $X$.
We call this the MJ-stratification.
In the similar way, for a morphism $\rho:{\mathcal X} \to Y$ as in the previous proposition, we can also define the locally closed subset 
$${\mathcal X}/Y(\delta)=\{x\in {\mathcal X} \mid \mldmj(x; {\mathcal X}_{\rho(x)})=\delta\},$$
and observe that $\{{\mathcal X}/Y(\delta)\}_\delta$ is a finite stratification of ${\mathcal X}$.
\end{rem}

\begin{lem}
\label{additive}
Let $X$ be a variety of dimension $d$. 
Let $V\subset W$ be two irreducible proper closed subsets of $X$
 and
$\eta_V$ and let $\eta_W$ be the generic points of $V$ and $W$, respectively.
Then the following inequality holds:
 \begin{equation}
 \label{add}
 \mldmj(\eta_V; X)\leq \mldmj(\eta_W; X)+\codim (V,W),
 \end{equation}
Here, 
if either $char\ k=0$ or Conjecture $ C_d $ holds, then we have the 
 equality in  $(\ref{add})$  for general $V$ in $W$.
I.e.,   there exists an open subset $U\subset W$ such that 
\begin{enumerate}
\item[]
if $\eta_V\in U$ holds 
for an irreducible closed subset $V\subset W$,\\
then the equality in (\ref{add}) holds.
\end{enumerate}
\end{lem}

\begin{proof} In \cite[Corollary 3.27, (ii)]{ir2} the inequality (\ref{add}) 
is proved for an arbitrary characteristic.
The equality in (\ref{add})  for general $V$ in $W$ for $char \ k=0$
is also proved in \cite[Corollary 3.27]{ir2}.
So it is sufficient to show the second statement under the assumption that Conjecture $C_d $
holds.
For $m\in \NN$ let $d_{mV}$ and $d_{mW}$ be the dimensions of a general fibers of 
$\pi_m: \pi_m^{-1}(V)\to V$ and $\pi_m: \pi_m^{-1}(W)\to W$, respectively.
Remember for every $m\in \NN$ the following hold:
$$s_m(X, \eta_V)=(m+1)d-\dim X_m(\eta_V)=(m+1)d-( \dim V+ d_{mV})$$
$$s_m(X, \eta_W)=(m+1)d-\dim X_m(\eta_W)=(m+1)d-( \dim W+ d_{mW}).$$

Then, as 
$d_{mV}\geq d_{mW}$ in general, we have
$$s_m(X, \eta_V)-s_m(X, \eta_W)=\codim (V, W) + d_{mW}-d_{mV}\leq \codim (V, W).$$
Let $U_m\subset W$ be the open subset such that the dimension of the fibers of a closed point by $\pi_m$  is the minimum.
Note that   $\eta_V\in U_m$ if and only if $d_{mV}=d_{mW}$, which is equivalent to
$$s_m(X, \eta_V)=s_m(X, \eta_W)+\codim (V, W).$$
Take the number ${N_d}$ in Conjecture  $C_d$  and let $U=\cap_{m=1}^{{N_d}}U_m$.
Then we should note that $U$ is an open dense subset of $W$, 
because it is the intersection of finite
number of open dense subsets.
Then for every $V\subset W$ with $\eta_V\in U$ we have
$$\mldmj(\eta_V, X)=\min_{m\leq N_d} s_m(X, \eta_V)=\min_{m\leq N_d} s_m(X, \eta_W)+\codim (V,W).$$
\end{proof}

\begin{lem} 
\label{characterization}
Let $\eta\in X$ be a point. If there exists a stratification $\{X(\delta)\}$ as in Remark \ref{stratification}.
Then the following hold:
\begin{enumerate}
\item[(i)] If  $\delta\geq \dim {X(\delta)^{(j)}}$ (\resp\ $\delta\geq \dim {X(\delta)^{(j)}}+1$ for every irreducible component $X(\delta)^{(j)}$ of the stratum $X(\delta)$ such that 
the closure of ${X(\delta)^{(j)}}$ contains 
$\eta$, then $X$ is MJ-log canonical (\resp\ MJ-canonical) at $\eta$;

\item[(ii)]  If either the base field $k$ is uncountable, or Conjecture  $C_d$  holds,
then the converse of $\rm(i)$ also holds.
\end{enumerate}
\end{lem}

\begin{proof}
First we note that a special case of the formula in Lemma \ref{additive} implies the following:

For a closed point $x$ of an irreducible closed subset $W\subset X$ with the 
generic point $\eta_W$,
\begin{equation}
\label{closed}
\mldmj(x;X)\leq \mldmj(\eta_W, X)+ \dim W.
\end{equation}

Let $W\subset X$ be an irreducible closed subset containing $\eta$ and let 
$\eta_W$ be the generic point of $W$.
For  (i), we assume $\delta\geq \dim X(\delta)^{(j)}$ and will show $\mldmj(\eta_W, X)\geq 0$.
Take an irreducible component $X(\delta)^{(j)}$ of the stratum $X(\delta)$  containing $\eta_W$ in its closure, then $\dim W\leq \dim X(\delta)^{(j)}$.
As a closed point $x\in X(\delta)$ has $\mldmj(x; X)=\delta$, it follows from (\ref{closed})
$$\mldmj(\eta_W;X)\geq \delta-\dim W\geq \delta-\dim X(\delta)^{(j)}\geq 0.$$

The proof for MJ-canonicity is similar, as we can replace the inequalities $\geq~0$ by
$\geq~1$. 
(Here, we note that for  MJ-log canonicity, we have only to prove 
$\mldmj(\eta, X)\geq 0$. 
The reason why we dare to prove $\mldmj(\eta_W, X)\geq 0$ for $W$ containing $\eta$
is because this proof works  for MJ-canonicity by just shifting the number.)

For the proof of (ii), we should note that there are closed points $x\in X(\delta)^{(j)}$
contained by the closure of $\eta$ such that
the following equality holds:
\begin{equation}\label{=}
\mldmj(x;X)=\mldmj(\eta', X)+\dim X(\delta)^{(j)},
\end{equation}
where $\eta'$ is the generic point of $X(\delta)^{(j)}$.
Actually the first case (uncountable base field case) is proved in \cite{ir2}
and the second case (Conjecture $C_d$  holds) is proved in Lemma \ref{additive}.
By the formula (\ref{=}), the assumption that $X$ is MJ-log canonical yields that 
$$\delta-\dim X(\delta)^{(j)}=\mldmj(x;X)-\dim X(\delta)^{(j)}=\mldmj(\eta', X)\geq 0.$$
The proof for MJ-canonicity is similar, as we can replace the inequalities $\geq 0$ by
$\geq 1$.
\end{proof}

The following global statement follows immediately from the local statement, Lemma \ref{characterization}.

\begin{cor} 
\label{global}
Assume a variety $X$ has the stratification as in Remark \ref{stratification}.
If the base field $k$ is uncountable or Conjecture  $C_d$  holds,
then a variety $X$ has MJ-log canonical (\resp\ MJ-canonical) singularities if and only if $$\delta\geq \dim X(\delta)\ (\mbox{\resp}\ \ \delta\geq \dim X(\delta) +1).$$

\end{cor}

Now we will prove (P{MJ}1).
\begin{prop}[Openness of MJ-log canonicity/MJ-canonicity]
\label{open}
Assume Conjecture  $C_d$  holds.
Let $X$ be a $d$-dimensional variety.
If $(X, \eta)$ is an MJ-log canonical (\resp\ MJ-canonical) singularity, then there is an open neighborhood
$U\subset X$  of $\eta$ such that $U$ have MJ-log canonical (\resp\ MJ-canonical) singularities at every point of $U$.
\end{prop}

\begin{proof}
     As we assume ($C_d$), there is the stratification $\{X(\delta)\}$ 
     on $X$ as in Remark \ref{stratification}.
     Let $Z$ be the union of the closures of irreducible components  
     $X(\delta)^{(j)}$ of strata 
      $X(\delta)$'s such that $\eta\not\in\overline{(X(\delta)^{(j)})}$.
      Then $U=X\setminus Z$ is an open neighborhood of $\eta$.
        If $(X, \eta)$ is MJ-log canonical, then the MJ-stratification $\{U_\delta
        =U\cap X(\delta)\}$ satisfies $
\delta\geq \dim U_\delta$ for every stratum 
$U_\delta$, by Lemma \ref{characterization}. 
By Corollary \ref{global}, this shows that $U$ has MJ-log canonical singularities.
For MJ-canonicity, the proof is similar.
\end{proof}

Next we prove (PMJ2) in the following:

\begin{prop}[Stability under a deformation]
 Assume Conjecture  $C_d$  holds. Let $\rho: \mathcal X \to \Delta$ be a surjective morphism of a variety to a smooth curve $\Delta$  with the equidimensional reduced fibers of dimension $d$. Denote the fiber of this morphism 
of a point $t\in \Delta$ by $\mathcal X_t$.  If $(X, x_0)=(\mathcal X_0, x_0)$ for $0\in \Delta$ 
is MJ-log canonical (\resp\ MJ-canonical), then there are open neighborhoods $\Delta'\subset
\Delta$ and ${\mathcal U}\subset \mathcal X$ of 0 and $x_0$, respectively, such that  
all fibers of ${\mathcal U}\to \Delta'$ have MJ-log canonical (\resp\ MJ-canonical) singularities.
\end{prop}

\begin{proof} By the assumption and Proposition \ref{open}, we may assume that
the fiber ${\mathcal X_0}$ has MJ-log canonical (\resp\ MJ-canonical) singularities,
by replacing $\mathcal X$ by a sufficiently small neighborhood around $x_0$.
We use the notation in Remark \ref{stratification}.
As the relative MJ-stratification $\{{\mathcal X}/\Delta(\delta)\}_{\delta}$ has a finite
number of irreducible strata, the set 
$$B=\left\{ t\in \Delta \mid t\neq 0, {\mbox {there\ is \ an\ irreducible\ component\ }} Z\  {\mbox{in\ }}
  {\mathcal X}/\Delta(\delta) {\mbox{such\ that\ } } Z\subset {\mathcal X}_t\right\}$$
  is a finite set.
  Replacing $\Delta$ by $\Delta'=\Delta\setminus B$,
  we may assume that every irreducible component of the strata ${\mathcal X}/\Delta(\delta)$ is dominating $\Delta$.
  For the statement of MJ-log canonicity, we have only to prove that
  $$\dim {\mathcal X}_t(\delta)\leq \delta$$
  for every  $\delta=-\infty, 0, 1,\ldots, d$ and $t\in \Delta$.
  Here, we note that ${\mathcal X}_t(\delta)=({\mathcal X}/\Delta(\delta))\cap {\mathcal X}_t$.
  Take an irreducible component $Z \subset
  {\mathcal X}/\Delta(\delta)$.
  If $\overline{Z}\cap {\mathcal X}_0=\emptyset$, then we replace ${\mathcal X}$
  by an open subset ${\mathcal X}\setminus \overline{Z}$.
  By this procedure we may assume that 
  $\overline{Z}\cap {\mathcal X}_0\neq \emptyset$
  for every irreducible component 
  $Z \subset
  {\mathcal X}/\Delta(\delta)$.
  
  Then consider the restriction $\rho': \overline{Z}\to\Delta$.
  Take an irreducible component $Z^{(i)}$ of ${\rho'}^{-1}(0)=\overline{Z}\cap {\mathcal X}_0$.
  Then  the generic point of $Z^{(i)}$ is contained in ${\mathcal X}_0(\delta_i)$ for some 
  $\delta_i\leq \delta$  by  the lower semi-continuity of 
  $\mldmj$ proved in Proposition \ref{lower}.
    By the assumption that ${\mathcal X}_0$ has MJ-log canonical singularities,
  we obtain $$\dim Z^{(i)}\leq \dim {\mathcal X}_0(\delta_i)\leq \delta_i
  \leq \delta.$$
  
  Now deleting finite number of points from $\Delta$ we may assume that 
  $$\dim {\rho'}^{-1}(t)\leq \dim {\rho'}^{-1}(0)\leq \delta.$$
  For each irreducible component of $ {\mathcal X}/\Delta(\delta)$ we have the same 
  inequality as above by deleting finite points of $\Delta$ and the number of such irreducible components are finite, 
  we obtain a non-empty open subset $\Delta'\subset \Delta$ such that
  
  $$\dim {\mathcal X}_t(\delta)\leq \delta$$
  holds for every $\delta=-\infty, 0, \ldots, d$ and for every $t\in \Delta'$.
  
  For the  MJ-canonicity, the proof goes parallel  as above. 
\end{proof}

\vskip2truecm
\section{Some affirmative cases}
\noindent
In this section we will show some affirmative cases for our conjectures.
We start with a simple observation:
\begin{prop} 
\label{first}
Conjecture $C_{d,d}$  holds for every $d\geq 1$ and $C_{d,d-1}$ holds for $d\geq 2$.
($C_{1,0}$ will be treated in the next section.)
We can take $N_{d,d}=1$ and $N_{d, d-1}=5$.

\end{prop}

\begin{proof} For a variety $X$ of dimension $d$ and a closed point $x\in X$,
the inequality 
$\mldmj(x; X)\geq d$ holds if and only if $s_m(X,x)\geq d$ for every $m\in \NN$.
In particular for $m=1$, the condition   $s_1(X,x)=2 d-\dim X_1(x) \geq d$
implies that $   \emb(X, x)=\dim X_1(x)\leq d$.
Then the equality $$   \emb(X, x)=\dim X_1(x)= d$$  must hold, which yields that $(X,x)$ is non-singular.
On the other hand, we know that $\mldmj(x; X)= d$ for non-singular $(X,x)$ (see for example
\cite[Corollary 3.17]{ir2}).
This shows Conjecture $C_{d,d}$ holds and $N_{d,d}=1$.
For $d\geq 2$, 
Conjecture $C_{d,d-1}$ was proved  and  $N_{d,d-1}=5$ is showed in \cite{ir2}.

\end{proof}

\vskip.3truecm

Now we will show some affirmative cases for Conjecture $C_d$.
\vskip.5truecm
\noindent

{\bf  Non degenerate {\hs}s.}

\begin{defn}\label{Newton}
Let $M = \ZZ^{d+1}$ and $M_{\RR} = M\otimes_{\ZZ}\RR\simeq \RR^{d+1}$.
For a finite set $\Lambda\subset M$, we define the {\it Newton polygon generated by $\Lambda$ }
 as follows
\begin{eqnarray*}
	\Gamma_\Lambda := \hbox{convex\ hull\ of\ } \biggl( \bigcup_{{\bf{m}}\in \Lambda}
	({\bf{m}}+\RR^{d+1}_{\ge 0})\biggr)\ 
\end{eqnarray*}
Let $\Lambda_0\subset \Lambda$ be the subset consisting of the vertices of 
$\Gamma_\Lambda$,
then, $\Gamma_\Lambda=\Gamma_{\Lambda_0}$.
We call $\Lambda_0$ the minimal generator set of $\Gamma_\Lambda$.

In particular, for a polynomial $f\in k[X_0,\cdots, X_d]$ we define $\Gamma_+(f)$, the {\it Newton polygon of
$f$} as follows:
For ${\bf{m}} =(m_0,\cdots ,m_d)\in M$ denote  $X^{{\bf{m}}}=X_0^{m_0}\cdots X_d^{m_d}$.
By using this expression we represent $f$ as  
$f = \sum_{{\bf{m}}\in M}a_{\bf{m}}X^{\bf{m}}$, ($a_{\bf{m}}\in k$).
Let $\Lambda= \{{\bf{m}}\in M\mid a_{\bf{m}}\ne 0\}$ and 
define $\Gamma_+(f)=\Gamma_\Lambda$.
\end{defn}

\vskip1\baselineskip

\begin{defn}\label{theorem:CDv} 
Under the notation above, $f$ is called {\it non-degenerate}
if every  face $\gamma$ of $\Gamma_+(f)$ the equations $\frac{\partial f_\gamma}{\partial X_i} = 0$ $(i = 0,\cdots, d)$ do not have common zeros on $(\AA^1_k\setminus \{0\})^{d+1}$, where 
 $f_\gamma := \sum_{\bf{m}\in \gamma}a_{\bf{m}}X^{\bf{m}}$.
 
 We say that $f$ is {\it non-degenerate with respect to compact faces}
 if for every compact face $\gamma$ the condition above holds.
\end{defn}

\begin{rem}
Note that non-degenerate hypersurfaces are general among all hypersurfaces 
with a fixed Newton polygon.

It is well known that a non-degenerate hypersurface has an embedded 
log-resolution by a toric birational transformation  in any characteristic.
(See, for example, \cite[III, Proposition 1.3.1]{o}. This proposition is stated under 
the base field is $\CC$, however the proof works for any algebraically closed field.)
We also note that if a hypersurface $X$ has an isolated singularity at $0$ and
defined by a non-degenerate polynomial with respect to compact faces,
then it also has an embedded log-resolution by a toric birational transformation  in any characteristic.

\end{rem}
By \cite[Lemma 5.4]{itoric} we have the following formula:

\begin{prop}\label{toricformula}
 Let $N$ be the dual of $M$, $\sigma\subset \RR\otimes_{\ZZ}M\simeq
\RR^{d+1}$ be the 
positive quadrant $(\RR_{\geq 0})^{d+1}$ and $\sigma^o$ be the interior of $\sigma$. 
Let $X\subset \AA^{d+1}$ be the hypersurface defined by a non-degenerate polynomial $f$.
For an element $\pp\in N$ we define $\langle \pp, \Gamma_+(f)\rangle=\min \{\langle \pp, \mm\rangle \mid \mm\in \Gamma_+(f)\}$.
We denote the point $(1,1,\ldots, 1,1)\in M$ by $\bf 1$.
Then,
$$\mldmj(0, X)=\mld(0, X)=\inf_{\pp\in\sigma^o\cap N}(\langle \pp, {\bf 1} \rangle-
\langle \pp, \Gamma_+(f)\rangle).$$
\end{prop}
Let us denote by $E_\pp$ the toric  prime divisor over $A=\AA^{d+1}$ corresponding to the 1-dimensional cone $\pp\RR_{\geq 0}$.
Then, $\langle \pp, {\bf 1} \rangle=k_{E_\pp}+1$ and $\langle\pp, \Gamma_+(f)\rangle=\val_{E_\pp}(f)$.

\begin{defn}
\label{polygon}
By the proposition above, for a non-degenerate hypersurface $X\subset \AA^{d+1}$,
$\mldmj(0,X)$ depends only on the Newton polygon $\Gamma_+(f)$.
Therefore we define the minimal log discrepancy for a Newton polygon $\Gamma$ by
$$\mld \Gamma= \inf_{\pp\in\sigma^o\cap N}(\langle \pp, {\bf 1} \rangle-
\langle \pp, \Gamma\rangle).$$

We say that $\pp$ computes $\mld\ \Gamma$ if 
$$\langle \pp, {\bf 1} \rangle-
\langle \pp, \Gamma\rangle=\mld\ \Gamma\geq 0,\ \ \ {\mbox{or}}$$
 $$\langle \pp, {\bf 1} \rangle-
\langle \pp, \Gamma\rangle<0, \ \ \mbox{when}\ \mld\ \Gamma=-\infty.$$
Here, we note that for a given Newton polygon $\Gamma$ we can find a non-degenerate 
polynomial $f$ such that $\Gamma=\Gamma_+(f)$,
and ``$\pp$ computes $\Gamma$" is the same as ``$E_\pp$ computes $\mldmj(0, X)$"
for the \hs \  $X$  defined by $f$.

\end{defn}

\begin{thm} 
\label{non-deg} 
For every $d\geq1$
Conjecture $C_d$ holds  for non-degenerate \hs{s} $X$
and the origin $0\in X$.
 \end{thm}
 
 \begin{proof} In the proof we use the notation of Proposition \ref{toricformula}.
 As the possible values of $\mldmj(0,X)$ are finite,
 it is sufficient to fix $\delta$ ($ \delta=-\infty, 0,\ldots, d$) and to prove
 a contradiction under the assumption that
 $\nu(X,0)$ is unbounded among non-degenerate hypersurface singularities $(X,0)$ satisfying
 $\mldmj(0,X)=\delta$.

  Let $\{\Gamma_j\}_j$ be an infinite sequence of Newton polygon with
  $\mld(\Gamma_j)=\delta$, such that $\nu_j:=\nu(X_j,0) \to \infty \ (\ j\to \infty)$,
  where $\Gamma_j$ is the Newton polygon of a non-degenerate hypersurface
  $X_j$.
 
 For points  $\ba, \bb\in\sigma\cap M$ we define the relation $\ba<\bb$,  if either 
 \begin{enumerate}
 \item[]  $|\ba|<|\bb|$, where $|\ba|$ is the sum of all coordinates of $\ba$, or
 \item[]  $|\ba|=|\bb|$ and $\ba<\bb$ lexicographically.
 \end{enumerate}
 
 We give numbers to all vertices  of each Newton polygon $\Gamma_j$
  according to the order:
 $$\ba_1(\Gamma_j) \leq \ba_2(\Gamma_j)\leq\cdots.$$
 
 Then we may assume that $|\ba_1(\Gamma_j)|\leq d+1$ for infinitely many $j$.
 Indeed, if there is an infinite subsequence such that 
 $|\ba_1(\Gamma_j)|\geq d+2$, then the multiplicity of the defining function $f_j$ of $X_j$ 
 at 0 
 is bigger than $d+1$, therefore by Lemma \ref{truncation} we obtain 
 $$s_{d+1}(X_j,0)=d(d+2)-\dim (\AA^{d+1})_{d+1}(0)=d(d+2)-(d+1)^2=-1 <0$$
 for each such $j$.
This implies that  $\delta=\mld(\Gamma_j)=-\infty $ and $\nu_j$ is bounded by $d+2$ for the subsequence, 
 which is  a contradiction to the assumption that 
   $\nu_j:=\nu(X_j,0) \to \infty \ (\ j\to \infty)$.
  Therefore, we may assume that there is an infinite subsequence 
  such that  $\ba_1(\Gamma_j)\leq d+1$ for all $j$ in the sequence.
  
  Here, as there are only finitely many  points $\ba\in \sigma\cap M $ with $|\ba|\leq d+1 $,
  there is an infinite subsequence of $\{\Gamma_j\}$ such that all $\ba_1(\Gamma_j)$
  are common. 
  Let them be $\ba_1$.
  
  Next, if $|\ba_2(\Gamma_j)|$ is bounded for infinitely many $j$ among the subsequence
  obtained above,
  then, in the same way as above we take the infinite subsequence with 
  the common $\ba_2(\Gamma_j)=:\ba_2$.
  We perform this procedure successively.
  But this procedure stops at a finite stage.
  Because, if not, then we obtain an infinite strictly increasing sequence of polygons: 
  $$P_1\subset P_2\subset\cdots,$$
  where $P_i$ is the polygon generated by $\ba_1,\ldots, \ba_i$.
  This means that there is an infinite strictly increasing sequence of monomial ideals 
  in a Noetherian algebra   $k[\sigma\cap M]$, which is a contradiction.
  
  Now, we can assume that there exists $m\in \NN$ such that there is an infinite
  subsequence $\{\Gamma_j\}$ with the common $\ba_1,\ba_2,\ldots, \ba_m$
  but $\ba_{m+1}(\Gamma_j)$ is not bounded for any infinite subsequences.
  Then, for any $D>0$ there is an infinite subsequence $\{\Gamma_j\}$
  such that $|\ba_{m+1}(\Gamma_j) | >D$ for all $j$.
  Let $\Gamma$ be the polygon generated by $\ba_1,\ba_2,\ldots, \ba_m$
  and $n=\nu(\Gamma)$.
  Then, we may assume that,  for all $j$, $|\ba_{m+1}(\Gamma_j) | \gg n$  so that 
  there is no lattice point on the faces of $\Gamma_j$ containing $\ba_{m+1}(\Gamma_j)$ .
  Let $f_j$ and $f$ be  non-degenerate polynomials with the Newton polygon $\Gamma_j$
  and $\Gamma$, respectively.
  We may assume that the monomials of these polynomials are only on the Newton boundary,
  because the minimal log discrepancies depend only on the Newton boundaries.
  By Lemma \ref{truncation}, we have 
  $$s_{n-1}(f_j)=s_{n-1}(f)=\mld(f)$$ for all $j$.
  Therefore $\mld(\Gamma_j)\leq \mld(\Gamma)$.
  But $\Gamma\subset \Gamma_j$ yields $\mld(\Gamma)\leq \mld(\Gamma_j)$
  by the Definition \ref{polygon}.
  Therefore, we obtain that $\mld(\Gamma_j)=\mld(\Gamma)=s_{n-1}(\Gamma_j)$ 
  for all $j$, which implies that $\nu_j\leq n$, a contradiction.
 
 \end{proof}

\begin{rem} 

\begin{enumerate}
\item
The proof of \cite[Theorem 1.4]{mn} can be interpreted into the discussion of
Newton polygon and 
we can prove the above theorem  in the same way as \cite[Theorem 1.4]{mn} using Maclagan's 
result.

\item
 When the base field is of characteristic 0,
by the following lemma we can reduce Conjecture $D_d$ to
 the  hypersurface case.
 Moreover, by Theorem \ref{non-deg},  the conjecture for characteristic 0 case is reduced to  degenerate hypersurfaces.
\end{enumerate}
\end{rem} 

\begin{prop}[{\cite[Theorem 3.1]{in}} ]  Let $0\in X\subset A$ be a closed point on a closed subvariety $X$ in a non-singular affine variety $A$ over a base field of characteristic 
zero. 
Then, there exists a hypersurface $0\in H\subset A$ such that

$$\mldmj(0, X)=\mldmj(0, H).$$ 
Moreover, there exists a log-resolution $A'\to A$ of $X$ and $H$ 
such that the prime divisors on $A'$ computing $\mldmj(0, X)$ are the same as 
those computing $\mldmj(0, H).$
\end{prop}

   In order to study a general hypersurface singularity, the following lemma is useful
   to prove the conjecture in a special case:
   
\begin{lem}\label{reduction/non-degenerate}
  Let $X\subset \AA^{d+1}$ be a hypersurface defined by a 
    polynomial $f$ with a singularity at $0$.
   Let    $\Gamma$ be a Newton polygon containing $\Gamma_{+}(f)$.
   If $\bold1=(1,1,\ldots, 1) \not\in \Gamma$, then $\mldmj (0; X)=-\infty$ and a toric divisor 
   $E_\pp$  computing $\mld \Gamma$ computes $\mldmj (0; X)$.
   In particular, if $\tilde f$ is a non-degenerate function with the Newton 
   polygon $\Gamma_+(\tilde f)=\Gamma$ and $\widetilde X$ is the hypersurface
   defined by $\tilde f$, then $\nu(X,0)\leq \nu(\widetilde X, 0) $ which is uniformly
   bounded on the dimension $d$.
\end{lem}

\begin{proof} 
      By the formula in Proposition \ref{toricformula},
      the assumption 
       $\bold1=(1,1,\ldots, 1) \not\in \Gamma$ implies $\mldmj(0; \widetilde X)=\mld \Gamma
       =-\infty$.
       Let $\pp$ compute $\mld\Gamma$, then, by $\Gamma_+(f)\subset \Gamma$, we
       have $$\val_{E_\pp}f\geq\langle \pp, \Gamma_+(f)\rangle\geq \langle \pp, \Gamma\rangle
       > \langle \pp, \bold 1\rangle,$$
       which implies that $\mldmj(0; X)=-\infty$ and $E_\pp$ computes it.
       
       For the second statement, note that $\nu(X,0)=\mu(X,0)$ and
        $\nu(\widetilde X,0)=\mu(\widetilde X,0)$.
        Take a prime divisor $E_\pp$ computing $\mldmj(0;\widetilde X)$ such that 
\begin{equation}\label{infinity}     
       k_{E_\pp}-\mu(\widetilde X,0)+1<0.
\end{equation}       
        Then, as $E_\pp$ also computes $\mldmj(0; X)$ by the discussion above, 
        we can see that  the inequality (\ref{infinity}) implies 
        $$\mu(X,0)\leq \mu(\widetilde X, 0)$$
        by the minimality of $\mu(X,0)$.
\end{proof}

\vskip.5truecm
\noindent
{\bf Singularities of maximal type.} 

      The second case that ($C_d$) holds
       is a $d$-dimensional singularity $(X,x)$ of maximal type 
      under the assumption that   Conjecture  $C_{d-1}$ holds for $(d-1)$-dimensional ``scheme".
      Here, note that we use a bit stronger conjecture $C_{d-1}$ for 
      not only  $(d-1)$-dimensional varieties but also $(d-1)$-dimensional
       schemes.
      This result is used in the next section.
      
     First, we give the definition of a singularity  of maximal type.

\begin{defn}
\label{maxtype}
   Let $X$ be a $d$-dimensional variety over $k$, $x\in X$ a closed point with 
   $\emb(X,x)=N$ and 
  let $c=N-d$.
  Let $X\subset A$ be a closed immersion around $x$ into a non-singular variety $A$ with 
  $\dim A=N$ and $I_X$ the defining ideal of $X$ in $A$.
  We define   $$\ord_xI_X=\min\{\mult_x f \mid f\in I_X\}.$$
   We say that the singularity $(X,x)$ is of maximal type 
   if  $$N=c\cdot\ord_xI_X.$$
  
\end{defn}

The following shows the status of a singularity of maximal type.

\begin{lem} 
\label{max}
   Under the notation in the above definition, 
   if a singularity $(X,x)$ is MJ-log canonical, then 
   $$N\geq c\cdot\ord_xI_X.$$
   If  the singularity $(X,x)$ is also of maximal type,
   then $\mldmj(x;X)=0$.

\end{lem}

\begin{proof}  Take the blow-up $A'\to A$ at the closed point $x$ and
   let $E$ be the exceptional divisor.
   Then, the log discrepancy $a(E; A, I_X^c)=N-c\cdot \ord_xI_X\geq 0$,
   because $(A, I_X^c)$ is log canonical by the assumption that $(X, x)$ is MJ-log
   canonical.
   This shows the first statement.
   If the singularity is moreover of maximal type, then $a(E;A , I_X^c)=0$,
   which shows that $\mldmj(x;X)=0$.
\end{proof}

  In order to give the basic lemma, we need to give a little generalization of definitions.
  
  
   \begin{defn} Fix $d\in \NN$.
   Let $Y$ be a $d$-dimensional scheme over $k$ and $\eta\in Y$ a  point.
   For $m\in \NN$, define a function $s_m(Y,\eta)$ as follows:
   $$s_m(Y,\eta)=(m+1)d-\dim Y_m(\eta).$$
   And also we define minimal MJ-log discrepancy 
   $$\mldmj(\eta;Y)=\inf_{m\in \NN}s_m(Y,\eta).$$
  We say that $(Y,\eta)$ is MJ-log canonical at a point $y\in Y$,
   if $\mldmj(\eta;Y)\geq 0$
  for every $\eta\in Y$ such that $\overline{\{\eta\}}\ni y.$
  Note that this is equivalent to 
  $$\mldmj(y;Y)\geq 0.$$

  The conjectures $C_d$, $D_d$ and $U_d$ are for
  pairs of $d$-dimensional varieties and closed points on them.
 Here, we extend these conjectures for pairs of $d$-dimensional schemes and closed points.

   We say that Conjecture $C_{d,i}$ holds for  $d$-dimensional schemes
    if there exists $N_{d,i}\in \NN$ depending only on $d$ and $i$ such that
\begin{enumerate}
\item[]    
    for every  pair $(Y,y)$ consisting of a $d$-dimensional scheme $Y$ and a closed point
    $y\in Y$, 
if $\mldmj(y;Y)<i$, then there exists $m\leq N_{d,i}$
with the property 
$s_m(Y,y)< i$.
\end{enumerate}

   We extend ($C_d$), ($D_d$) and ($U_d$) to this class of singularities
  in similar ways.
    \end{defn}    

\begin{prop}\label{log-canonical-Y} Let $Y$ be a $d$-dimensional scheme.
\begin{enumerate}
\item[(i)]  If $Y$ is embedded into a non-singular variety $A$ of dimension $N$
 as a closed subscheme, then for a closed point $y\in Y$,
     \begin{equation}
     \label{scheme}
     \mldmj(y;Y)=\mld(y;A, I_Y^{N-d}).
\end{equation}

\item[(ii)] If a $d$-dimensional scheme $Y$ is MJ-log canonical at every point of $Y$,
      then, $$\dim Y_m\leq (m+1)d.$$
\end{enumerate}
\end{prop}     

\begin{proof} The statement (i) follows by interpreting the definition of 
   $\mldmj(y;Y)$ in terms of contact loci.
   For (ii), take an irreducible component $W$ of $Y_m$ which gives the
   dimension of $Y_m$.
   Let $\eta\in Y$ be the generic point of the image $\pi_m(W)$.
   As $Y$ is MJ-log canonical at $\eta$ we obtain 
   $$0\leq \mldmj(\eta;Y)\leq (m+1)d-\dim (\pi_m)^{-1}(\eta),$$
   which yields $$\dim Y_m\leq (m+1)d.$$

\end{proof}

   The following is a slight generalization of \cite[Lemma 3.4]{st} and
    will be used to reduce the conjecture for singularities of maximal type to the case of    
    singularities defined by homogeneous ideals.

\begin{lem} 
\label{shibata-tam}
    Let $(X,x)$ be a 
 singularity on a $d$-dimensional variety $X$ of maximal type with the   
     embedding
     $X\subset A$ into a non-singular variety $A$ of dimension $N=\emb(X,x)$.
     Let $I_X=(f_1,\ldots, f_r)$ be the defining ideal of X in $A$.
     Let $\{f_1,\ldots, f_s\}$ $(s\leq r)$ be the subset of the generators such that 
     $\mult_xf_i= \alpha:=\ord_xI_X$ for $i=1,\ldots, s$.
     Define $J=(\ini f_1,\ldots, \ini f_s)$ and let $J$ define a closed subscheme 
     $Y     \subset \AA^N
     =\spec k[x_1,\ldots, x_N]$, where $\{x_1,\ldots, x_N\}$  is a regular system of      
     parameters of
     $A$ around $x$.

     Then,  for any integer $m\geq \alpha$, 
     $$s_m(X,x)\geq s_m (Y, 0)\geq s_{(\alpha+1)m-\alpha^2}(X, x).$$

\end{lem}

\begin{proof} 
Let $I_m$ be the defining ideal of $X_m(x)$ in $A_m(x)$ and
let $J_m$ be the defining ideal of $Y_m(0)$ in $\AA^N_m(0)$.
We can identify $A_m(x)$ and $\AA^N_m(0)$ as follows:
 $$A_m(x)=\AA^N_m(0)=\spec k[\bx^{(1)}, \bx^{(2)},
\ldots,\bx^{(m)}],$$
where $\bx^{(j)}$ denotes the collection of the coordinates $x_1^{(j)},\ldots, x_N^{(j)}$.
Under this identification, we have 
 $J_m\subset \ini I_m$, therefore $\height J_m\leq \height (\ini I_m)=\height I_m$,
which yields the first inequality in the theorem.

For the second inequality, we give a technical definition.
For any integers $l\geq 2$ and $m>l$ and a polynomial $f\in k[\bx^{(1)}, \bx^{(2)},
\ldots,\bx^{(m)}]$,
the symbol $ \overline{f} $  denotes  the polynomial in $k[\bx^{(l)}, \bx^{(l+1)},
\ldots,\bx^{(m)}]$ substituting $\bx^{(1)}= \cdots= \bx^{(l-1)}=0$ into $f$.

\noindent
{\bf Claim 1.} For a monomial $h\in k[x_1,\ldots, x_N]$ of degree $\geq \alpha+1$,
   it follows that
   $$\overline{h^{(i)}}=0, \ \ (1\leq i\leq (\alpha+1)l-1),$$
  where $h^{(i)}$ is as in (\ref{derivation}) in Proposition \ref{truncation},
   i.e., 
   
   \begin{equation}
\label{**}
h\left(\sum_{j=0}^m x_1^{(j)}t^j,\ldots, \sum_{j=0}^m x_N^{(j)}t^j\right)
\equiv \sum_{j=0}^m
 h^{(j)}t^j\ (\mod (t^{m+1})).
 \end{equation}
   For the proof of the Claim 1, remind us that  $\deg h^{(i)}\geq \alpha+1$ and the weight $i$ of $h^{(i)}$ is less
   than $(\alpha+1)l$, there is a variable $x_v^{(u)}$ in $h^{(i)}$ with the weight 
   $u\leq l-1$.
   Therefore, substituting $\bx^{(1)}= \cdots= \bx^{(l-1)}=0$ into $h^{(i)}$,
   we obtain $\overline{h^{(i)}}=0$.
   
\noindent
{\bf Claim 2.} Let $h\in k[x_1,\ldots, x_N]$ be a polynomial with $\mult_0 h\geq \alpha$,
    where $\mult_0h$ is the degree of initial term of $h$ in 
    variables $x_1,\ldots, x_N$.
    Let $\ini h$ be the initial term of $h$.
    Then 
    $$\overline{h^{(i)}}=\overline{(\ini h)^{(i)}}, \ \ (1\leq i\leq (\alpha+1)l-1).$$
    Indeed, if $\mult_0h\geq\alpha+1$, then the both hand sides of above are zero
    by Claim 1.
    If $\mult_0 h=\alpha$, then again by Claim~1, we have that only the parts of degree $\alpha$
     can survive by substituting $\bx^{(1)}= \cdots= \bx^{(l-1)}=0$.
     This completes the proof of Claim~2.

  Now let the defining ideal of $X_n(x)$ in $A_{n}(x)$
  be $I_{n}$ for $n\in \NN$.
   Then, for $n=(\alpha+1)l-1$, we have
  $$I_{(\alpha+1)l-1}\subset (\bx^{(1)},\ldots, \bx^{(l-1)}, f_1^{(\alpha l)},\ldots,
   f_r^{(\alpha l)},\ldots, f_1^{((\alpha+1)l-1)},\ldots,  f_r^{((\alpha+1)l-1)}).$$
   
   Here, we denote $$I(\alpha, l)=\left(\overline{f_1^{(\alpha l)}},\ldots,
\overline{ f_r^{(\alpha l)}},\ldots, \overline{f_1^{((\alpha+1)l-1)}},\ldots,  \overline{f_r^{((\alpha+1)l-1)}}\right)\subset k[\bx^{(l)}, 
\ldots,\bx^{((\alpha+1)l-1)}].$$
 
   Then we obtain
 \begin{equation}\label{shiba}
 \height I_{(\alpha+1)l-1}\leq \height I(\alpha, l) +(l-1)N.
 \end{equation}
   Here, by Claim 2, we obtain that
   $$I(\alpha,l)=\left(\overline{\ini f_1^{(\alpha l)}},\ldots,
\overline{ \ini f_r^{(\alpha l)}},\ldots, \overline{\ini f_1^{((\alpha+1)l-1)}},\ldots,  \overline{\ini f_r^{((\alpha+1)l-1)}}\right)$$
$$=\left(\overline{\ini f_1^{(\alpha l)}},\ldots,
\overline{ \ini f_s^{(\alpha l)}},\ldots, \overline{\ini f_1^{((\alpha+1)l-1)}},\ldots,  \overline{\ini f_s^{((\alpha+1)l-1)}}\right)$$

   By the shift $x_v^{(u)}\mapsto x_v^{(u-(l-1))}$ of  variables, 
   the ideal $I(\alpha,l)$ becomes 
   $$\left( \ini f_1^{(\alpha)},\ldots, \ini f_s^{(\alpha)},\ldots, \ini f_1^{(\alpha+l-1)},
   \ldots, \ini f_s^{(\alpha+l-1)}\right)=J_{\alpha+l-1}.$$  
   Hence, $\height I(\alpha,l)=\height J_{\alpha+l+1}$.
   By these interpretations, the inequality (\ref{shiba}) yields
   $$ \height I_{(\alpha+1)l-1}\leq \height J_{\alpha+l-1} +(l-1)N.$$
   If we put $m=\alpha+l-1$, then we obtain
   
\begin{equation}\label{compare}
   \height I_{(\alpha+1)m-\alpha^2}\leq \height J_m+(m-\alpha)N,
\end{equation}   
   which will give the required inequality.
   Actually, to see this, remind us the definition of $s_m(X,x)$ and $s_m (Y,0)$
   to obtain:
   $$s_m (Y,0)=(m+1)d-mN+\height J_m=-cm+d +\height J_m,$$
   $$s_{(\alpha+1)m-\alpha^2}(X,x)=
   -c((\alpha+1)m-\alpha^2)+d+\height I_{(\alpha+1)m-\alpha^2}.$$
   Substituting these into (\ref{compare}) and by noting that $N=\alpha\cdot c$, we finally obtain
   $$s_m (Y,0)\geq s_{(\alpha+1)m-\alpha^2}(X,x).$$
\end{proof}

\begin{cor} Under the same notation and the assumptions as in the previous lemma,
 the following are equivalent:
\begin{enumerate}
\item $(X,x)$ is MJ-log canonical;
\item $Y$  is   of dimension $d$ at $0$ and
$\mldmj (0;Y)=0.$
\end{enumerate}
In these cases, $(Y,y)$ is also of maximal type.
\end{cor}

\begin{proof}
    If $(X,x)$ is MJ-log canonical, then  
     $s_m(X,x)\geq 0$, for every $m\in \NN$.
     Then, by Lemma \ref{max} and Lemma \ref{shibata-tam},
     we obtain that $\mldmj (0;Y)= 0$.
     By the definition of $Y$, we have $\dim Y\geq d$.
     Here, if $Y$ has an irreducible component $Y_1$ of dimension greater than $d$,
      then 
     $$\mldmj (0;Y)=\mld(0; \AA^N, I_Y^c)\leq \mld(0; \AA^N, I_{Y_1}^c)=-\infty,$$
     which is a contradiction to Lemma \ref{shibata-tam}.

    Conversely, we assume that $Y$ is a  scheme of dimension $d$ 
    and $\mldmj (0; Y)=0.$
    Then we have $s_m (Y,0)\geq 0$ for every $m\in \NN$.
    By Lemma \ref{shibata-tam}, we obtain $\mldmj(x;X)\geq 0$.
    
\end{proof}



The following shows a reduction step on $C_{d,i}$ for a singularity of maximal type
defined by  homogeneous polynomial with the same degree.
Here, we note that $i$ can be negative.

\begin{lem} 
\label{max-type-homo} Fix   integers $d\geq 2$, $i<d$ and $\alpha\geq2$.
  Let $\mathcal S$ be the set of pairs $(Y,0)$@consisting of a $d$-dimensional scheme $Y$  over $k$ and a closed point $0\in Y$ satisfying 
  \begin{enumerate}
  \item[(a)]
    $(Y,0)\subset (\AA^N,0)$ is of maximal type and
  \item[(b)]
  $(Y,0)\subset (\AA^N,0)$ is defined by homogeneous polynomials 
  of a common degree $\alpha$.
\end{enumerate}   
Assume Conjecture $ C_{d-1, i-1} $ holds true for the set of pairs consisting of 
   $(d-1)$-dimensional schemes over $k$ and  closed point on it.

      Then, Conjecture $ C_{d, i} $ holds true for $\mathcal S$ and 
      $N_{d, i}\leq \max \{\alpha,   N_{d-1,i-1}\}$.
\end{lem} 

\begin{proof} 
   Take $(Y,0)$ from $\mathcal S$.
       Then, by (a), we have $N=(N-d)\alpha$. Therefore it follows
\begin{equation}\label{express N}      
      N=(\alpha/(\alpha-1))d\leq 2d\ \mbox{ and\  also} \ \alpha\leq N\leq 2d.
\end{equation}      
    In order to prove $C_{d,i}$, assume $\mldmj(0; Y)<i$.
   By  (a), we see that the exceptional divisor $E_1$ of the blow-up $A'\to A=\AA^N$
   at the origin
   has log-discrepancy $a(E_1;A, I_Y^c)=0$,
   which implies $\mldmj(0;Y)\leq 0$.
   Therefore, it is sufficient to 
prove $C_{d,i}$ for  $i\leq 0$.
Henceforth in the proof we  assume that $i\leq 0$.

{\bf Case 1.} 
When $Y\setminus\{0\}$ is MJ-log canonical, 
     then we will show that there is 
\begin{equation}\label{alpha}     
    m\leq \alpha\ \mbox{such\  that}\ s_{m} (Y,0)< i.
\end{equation}    
     
    Indeed, let $f_1,\ldots, f_r$ be the homogeneous generators of $I_Y$ of degree $\alpha$,
    then, 
      for $j\geq \alpha$,
    $$f_l^{(j)}({\bf{0}} ,\bx^{(1)},\ldots,\bx^{(j)})\ \  \mbox{corresponds\ to\ }
      f_l^{(j-\alpha)}(\bx^{(0)},\ldots, \bx^{(j-  \alpha)}),$$
     by the shift of variables $x_v^{(u)} \mapsto x_v^{(u-1)}$,
    because $f_l$ ($l=1,\ldots, r$) are homogeneous of degree $\alpha$.
     Since $Y_m(0)$ is defined by $f_l^{(j)}({\bf{0}},\bx^{(1)},\ldots)$ 
     ($j=\alpha,  \ldots, m$) in $\spec k[\bx^{(1)},\ldots, \bx^{(m)}]$,
      we obtain that
      $$Y_m(0)\simeq\spec k[\bx^{(0)},\ldots, \bx^{(m-1)}]/\left(f_l^{(j-\alpha)}
      \left(\bx^{(0)},
      \ldots, \bx^{(j-\alpha)}\right)\right)_{j-\alpha=0,..,m-\alpha}$$

      $$\simeq \left[\spec k[\bx^{(0)},\ldots, \bx^{(m-\alpha)}]/\left(f_l^{(j-\alpha)}\left(\bx^{(0)},
      \ldots, \bx^{(j-\alpha)}\right)\right)_{j-\alpha=0,..,m-\alpha}\right]\times \AA^{(\alpha-1)N}$$
      $$\simeq Y_{m-\alpha}\times\AA_k^{(\alpha-1)N}.$$
      
      Hence we have
\begin{equation}
\label{local-global}
    \dim Y_m(0)=\dim Y_{m-\alpha}+(\alpha-1)N,
\end{equation}         
  
    Now,  take the smallest number $m$ satisfying $s_m (Y,0)< i$.
    Then 
\begin{equation}
\label{outside-MJ-log}    
    \dim Y_m(0)> d(m+1)-i.
\end{equation}     
   If $m > \alpha$, then  (\ref{express N}), (\ref{local-global}) and (\ref{outside-MJ-log}) yield
    $$\dim Y_{m-\alpha}>d(m-\alpha+1)-i.$$ 
   As $Y\setminus \{0\}$ is MJ-log canonical, 
   $\dim \left(Y_{m-\alpha} \setminus Y_{m-\alpha}(0)\right)\leq d(m-\alpha+1)$ by 
   Proposition \ref{log-canonical-Y},
   which yields $\dim Y_{m-\alpha}(0) > d(m-\alpha+1)-i$ and 
   therefore 
   $$s_{m-\alpha} (Y,0)< i,$$
   which is a contradiction to the minimality of $m$.
   Therefore we obtain (\ref{alpha}).
    
{\bf Case 2.}  
    When $Y\setminus \{0\}$ is not MJ-log canonical, then
    there is a number $$m\leq N_{d-1,i-1}\ \ \mbox{ such\  that}\ \ 
    s_m (Y,0)< i.$$
    Indeed, since $Y$ is an affine cone, for any open neighborhood $U$ of 
    the vertex $0\in Y$, there is a non-MJ-log canonical closed point $y\in U\setminus \{0\}$.
    We may think that $(Y, y)$ is isomorphic to $(Z\times \AA^1, z\times 0)$
    around $y$ for some $(d-1)$-dimensional scheme $Z$ and its closed  point $z$. 
    As we assume that Conjecture $ C_{d-1,i-1}$ holds, 
    there is a number $$m\leq N_{d-1,i-1}\mbox{\ \ such\  that}\ \ 
    s_m ( Z,z)< i-1.$$
    Therefore we obtain
    $$s_m (Y,y)=s_m (Z\times \AA^1, z\times 0)=s_m (Z,z)+1< i.$$
    Note that the vertex $0$ is in the closure of $z\times \AA^1$ and $s_m (Y,\ \ )$
    is lower semicontinuous, we obtain 
    $$s_m (Y,y)< i, \mbox{\ \ for}\ \ m\leq N_{d-1,i-1}$$
 as required.   

As the conclusion of    Case 1 and Case 2, we obtain that $C_{d, i}$ holds true and $N_{d, i}\leq 
\max \{\alpha, N_{d-1,i-1}\}$.
\end{proof}

\begin{prop}
\label{reduce}
 For an integer $i\leq d$.
 Assume Conjecture $ C_{d-1, i-1} $ holds true.
   Then, Conjecture $ C_{d, i} $ holds true for the category  $\mathcal S_\alpha$ of
   singularities  
   $(X, x)\subset (A, x)$ on varieties $X$ of maximal type with  $\alpha=\ord_xI_X$, then in this category we can take
   $$N_{d,i}= \max\{\alpha, (\alpha+1)(d-1-i)-\alpha^2, (\alpha+1)N_{d-1, i-1}-\alpha^2\}.$$
   
   In particular, if Conjecture $(C_{d-1})$ holds, then in the category of 
   $d$-dimensional singularities $(X,x)$ on varieties $X$ of maximal type,
   Conjecture $C_d$ holds.
\end{prop}

\begin{proof}
  Let $(X,x)\subset (A,x)$ be a singularity of maximal type with the dimension $d$.
  Let $\dim A=N=\emb(X,x)$.
  Let $(Y,0)\subset \AA^N$ be the scheme defined by homogeneous polynomials
  of degree $\alpha$ as introduced in Lemma \ref{shibata-tam}.
 
  When $\dim Y>d$,  then $\dim Y_m(0)\geq m\cdot \dim Y\geq m(d+1)$.
Then, it follows that 
$$s_m(Y,0)=(m+1)d-\dim Y_m(0)\leq (m+1)d-m(d+1)=d-m.$$
Therefore,  for $m=d-i+1$, we obtain $s_m(Y,0)<i$,
which shows $(C_{d,i})$ for the class of such $Y$ by taking $N_{d,i}= d-i+1$.

   Next consider  the case $\dim Y=d$.
  Note that $(Y,0)$ satisfies the condition (a), (b)  in Lemma \ref{max-type-homo}.
  By the assumption of the proposition and  Lemma \ref{max-type-homo}, 
  there is $N'_{d, i}$ such that  for some $m\leq N'_{d,i}$, $s_m(Y,0)< i$, if $\mldmj(0;Y)<i$.
  Here, by Lemma \ref{max-type-homo}, we can take $N'_{d,i}= \max \{\alpha,
  d-i+1, N_{d-1, i-1}\}$.
 
  By Lemma \ref{shibata-tam}, $C_{d,i}$ holds in the category of singularities of maximal type
  with $\alpha=\ord_xI_X$
  and we can take $$N_{d,i}=(\alpha+1)N'_{d,i}-\alpha^2=
  \max\{(\alpha+1) \alpha-\alpha^2, (\alpha+1)(d-i+1)-\alpha^2, (\alpha+1)N_{d-1, i-1}-\alpha^2\}.$$
  
  For the last statement, note that the bound $N_{d,i}$ in $\mathcal S_\alpha$ depends
  on $\alpha$.
  But 
 $d$-dimensional singularities of maximal type
  with $\alpha=\ord_xI_X$ have a bound $\alpha\leq 2d$.

\end{proof}

\vskip.5truecm
\section{Curve and surface cases}
\begin{thm}\label{curve}
In the category of connected schemes of dimension 1 over the base field $k$ of arbitrary 
characteristic, 
Conjecture $C_1$ (therefore Conjecture $D_1$ and $U_1$ also) holds and we can take $N_1=5$,  $M_1=4$ and $B_1=3$.
More precisely, we can take 
 $N_{1,1}=1$, $N_{1,0}=5$ and $N_{1,-1}=11$.
\end{thm}

\begin{proof} 
   Once we obtain the  bounds  $N_{1,1}$ and $N_{1,0}$, then the statement about $N_{1,-1} $ follows 
   from Proposition \ref{3equi}.
   We can prove ($C_1$) by direct calculations of the jet schemes of one dimensional schemes.
   But the simplest way for a proof of  the theorem is to show ($D_1$).
   Let $X $ be a 1-dimensional scheme over $k$ embedded into a non-singular variety $A$ 
   of dimension $\leq 2d=2$.
   Then, $X$ is a plane curve defined by one equation $f=0$.

  Conjecture ($C_{1,1}$) is already proved and $N_{1,1}=1$.
  So, we may assume that $\mult_xf\geq 2$.
     
{\bf Case 1.}   If $\mult_xf\geq 3$, the blow-up $A'\to A$ at the closed point $x\in X\subset A$ provides with the prime divisor $E_1$ over $A$ which computes $\mldmj(x;A, I_X)=-\infty$.
In this case $k_{E_1}=1$.

{\bf Case 2.}  Assume that $\mult_xf=2$. If $X$ is a normal crossing double point at $x$,
i.e., $f=x_1\cdot x_2$ in $\widehat\o_{A,x}=k[[x_1,x_2]]$,
then the blow-up $A'\to A$ at the closed point $x\in X\subset A$ provides with the prime divisor $E_1$ over $A$ which computes $\mldmj(x;A, I_X)=0$.

If $X$ is not  normal crossing double at $x$, then the exceptional divisor $E_3$ of the
third blow-up computes $\mldmj(x;A, I_X)=-\infty$.
In this case, $k_{E_3}=4$ and $b(E_3)=\widetilde b(E_3)=3$.

As a conclusion we obtain $(D_1)$ and $M_1=4$, $B_1=3$.
On the other hand, by the proof of Proposition \ref{C=D} we can take
$$N_1=M_1+1=5.$$
   
\end{proof}

Next we are going to prove Conjecture $C_2$.
Let us see some examples  used in the proof of $C_2$.

\begin{exmp}\label{example}
We observe some examples of Newton polygons $\Gamma$ such that
$\mld\ \Gamma=-\infty$. 
If a hypersurface $X\subset \AA^3$ is defined by a (not necessarily non-degenerate)
polynomial $f\in k[x_1,x_2,x_3]$ whose  Newton polygon $\Gamma_+(f)$
contained in the following $\Gamma_i$, then $\mldmj(0, X)=-\infty$. 
In the following,
by applying Lemma \ref{reduction/non-degenerate}, 
we   obtain a bound of $\nu(X,0)$ that is the minimal value $m$
such that $s_{m-1}(X,0)<0$. 

\begin{enumerate}

\item Let $\Gamma_1$ be the Newton polygon generated by three points
   $$(2,0,0), (0,5,0), (0,0,5) \in M=\ZZ^3.$$ 
   Then, for $\pp=(5,2,2)\in N=M^*$, we obtain 
   $$\langle \pp, {\bf 1}\rangle=9,\ \ \langle \pp, \Gamma_1\rangle = 10.$$
   Therefore, $$0>\langle \pp, {\bf 1}\rangle-\langle \pp, \Gamma_1\rangle\geq\mld\ \Gamma_1=-\infty.$$
   
   If $\Gamma_+(f)\subset \Gamma_1$ for a polynomial defining a hypersurface 
   $X\subset \AA^3$,
   we have
    $k_{E_\pp}=8 $ and $\val_{E_\pp}(f)\geq \langle \pp, \Gamma_1\rangle = 10$,
    hence $\mldmj(0, X)=-\infty$.
    We also have $\nu(X,0)\leq 10$.
      
\item We use the same notation as in (1).
     Let $\Gamma_2$ be  generated by 
   $(2,0,0), (0,3,0)$ and $ (0,0,7)$.
Then, for $\pp=(21,14,6)\in N$, we obtain 
   $$\langle \pp, {\bf 1}\rangle=41,\ \ \langle \pp, \Gamma_2\rangle = 42,  \ \ \mldmj(0, X)=-\infty,\ \ \mbox{and}\ \ 
   \nu(X,0)\leq 42.$$

\item
     Let $\Gamma_3$ be  generated by 
   $(2,0,0), (0,3,1), (0,0,5)$.
Then, for $\pp=(15,8,6)\in N$, we obtain 
   $$\langle \pp, {\bf 1}\rangle=29,\ \ \langle \pp, \Gamma_3\rangle = 30,\ \  \mldmj(0, X)=-\infty,\ \ \mbox{and}\ \ 
   \nu(X,0)\leq 30.$$  

\item
     Let $\Gamma_4$ be  generated by 
     $(2,0,0), (0,4,0), (0,0,5)$.
     Then, for $\pp=(10,5,4)\in N$, we obtain 
   $$\langle \pp, {\bf 1}\rangle=19,\ \ \langle \pp, \Gamma_4\rangle = 20, \ \ \mldmj(0, X)=-\infty,\ \ \mbox{and}\ \  \nu(X,0)\leq 20.$$   
\end{enumerate}   
   Note that the Newton polygon in (1) is contained in the polygon in (4).
   In this sense, the example (1) seems redundant.
   The reason why we take (1) as an example is because the valuation of $I_X$ at the
   prime divisor computing $\mldmj$ in (1) is smaller than that of (4).

\end{exmp}

\begin{thm} 
\label{surface}
If $char k\neq 2$, then Conjecture $C_2$ (therefore Conjecture $D_2$ and $U_2$ also) holds 
for $2$-dimensional schemes and  we can take $N_2= 41$, $M_2=58$ and $B_2=39$.
\end{thm}

\begin{proof}  As we saw, Conjecture $C_{2,2}$ and $C_{2,1} $ hold,
   the only problems are to show $C_{2,0}$ and to obtain the bound numbers, $ N_2, M_2$
   and $B_2$.  
   
  For $char k=0$, aside  the bound numbers, 
  Conjecture $C_{2,0}$ is proved in \cite{mn}.
  Actually ($C_{2,0}$) is translated into the following:
\begin{enumerate}
\item[]  
  For a non-singular point $(A,x)$ of $\dim A=N=3, 4$, there is an integer $M\in \NN$
  independent of the choice of 2-dimensional subscheme $X\subset A$ such that  if   the inequality 
  $$a(E;A, I_X^{N-2})\geq 0$$
  holds for every prime divisor $E$ over $A$ with the center at $x$ satisfying $k_E\leq M$,
 then $$\mld(x;A,I_X^{N-2})\geq 0$$
  holds.
  \end{enumerate}
  This is proved in \cite[Proposition 3.3]{mn} under a more general setting and it completes 
  the proof of ($C_2$)  for  characteristic $0$.
  \vskip.5truecm
  
 Here we give a proof which works for any characteristic   except for $ 2$
 and give bound values of $N_2$, $M_2$ and $B_2$.
 (As one sees in the proof, the values may not be optimal.)
  We divide the problem into 6 cases and will check each.
  Among these cases only the 6th case needs the condition that $char k\neq 2$ for the proof.
  
  {\bf Case 1.}    Assume $\emb(X,0)\geq 5$.

   In this case $s_1(X,0)=2\cdot 2-\dim X_1(0)=4-\emb(X,0)<0$, 
   therefore $\nu(X,0)\leq 2$. 
   On the other hand the exceptional divisor $E_1$ of the first blow-up computes $\mldmj(x;X)=-\infty$.
   
{\bf Case 2.}  Assume $\emb(X,0)= 4$ and $\ord_0I_X\geq 3$.

   In this case, $s_2(X,0)=3\cdot 2-\dim X_2(0)=6-8 < 0$,  therefore $\nu(X,0)\leq 3$.
   There is a prime divisor $E$ computing $\mldmj(x;X)=-\infty$@such that
   $k_E\leq 4$ and $b(E)\leq 2$.

{\bf Case 3.}  Assume $\emb(X,0)= 4$ and $\ord_0I_X=2$.

   In this case, $(X,0)\subset (\AA^4, 0)$ is of maximal type,
   because $$\codim (X,\AA^4)\cdot \ord_0I_X=2\cdot 2=4=\dim \AA^4.$$
   As we know in Theorem \ref{curve} that Conjecture $C_1$ holds for the category of schemes, 
  applying Proposition \ref{reduce}, we obtain that 
   $C_{2,0}$ holds and in this case, $$N_{2, 0}\leq   3\cdot 11-4 =29.$$
 Therefore, $\nu(X,0)\leq 30$.
 There is a prime divisor $E$ over $A$ computing $\mldmj(x;X)=-\infty$
 such that $k_E\leq 58$ and $b(E)\leq 28$.

{\bf Case 4.} Assume $\emb(X,0)=3$ and $\ord_0I_X\geq 4$.

In this case, $s_3(X,0)=4\cdot 2-\dim X_3(0)= 8-9 < 0$, therefore $\nu(X,0)\leq 4$.
There is a prime divisor $E$ over $A$ computing $\mldmj(x;X)=-\infty$
such that $k_E\leq 3$ and $b(E)\leq 3$.
    
{\bf Case 5.}  Assume  $\emb(X,0)=3$ and $\ord_0I_X= 3$.

In this case, $(X,0)\subset (\AA^3, 0)$ is of maximal type,
   because $$\codim (X,\AA^3)\cdot \ord_0I_X=1\cdot 3=3=\dim \AA^3.$$
   We know  Conjecture $C_1$ holds.
   Therefore, applying Proposition \ref{reduce}, we obtain that 
   $C_{2,0}$ holds and in this case, $$N_{2,0}\leq (3+1)\cdot 11-9=35.$$ 
   Therefore, $\nu(X,0)\leq 36$.
   There is a prime divisor $E$ over $A$ computing $\mldmj(x;X)=-\infty$
such that $k_E\leq 35$ and $b(E)\leq 34$.
   
{\bf Case 6.} Assume $\emb(X,0)=3$ and $\ord_0I_X= 2$.

    In this case, the singularity is a {\hs}   double point. So we let it be defined by
    $f\in k[[x,y,z]]=\widehat{\o}_{\AA^3, 0}$ such that $\mult_0f=2$.
    For simplicity, we denote $\AA^3$ by $A$.
    
     By Weierstrass Theorem for the formal power series ring
   $k[[x,y,z]]$ one has the presentation of $f$ as follows:
   $$f=x^2+xg(y,z)+h(y,z),$$
   where $g, h\in k[[y,z]]$.
   As $char k\neq 2$, we can make the Tschirnhausen transformation
   $x+(1/2)g = x'$ to have
   $$f=x'^2+h'(y,z).$$
   Hence, by the coordinate change of the formal power series ring 
   $k[[x,y,z]]$, we may write
\begin{equation}\label{f}
   f=x^2+h(y,z),
\end{equation}
   where $h\in k[[y,z]]$.
   Here, we may assume that $\mult_0h\geq 3$. 
   Because if $\mult_0h=2$,
   then by \cite{ir2}, $X$ has the singularity with $\mldmj(0; X)=1$ and 
   $C_2$ is already proved in this case and $\nu(X,0)\leq 6$.
   
   On the other hand, we may also assume that $\mult_0h\leq 4$.
   Because if $\mult_0h\geq 5$, then the toric divisor $E_{\pp}$ corresponing
   to $\pp=(5,2,2)$ satisfies
   $a(E_\pp, A, I_X)<0$ as we have seen in Example \ref{example}, (1).
   Therefore by Lemma \ref{reduction/non-degenerate}, we have $\nu(X,0)\leq 10$.
   \vskip.5truecm
   Now we have to consider only two classes, $\mult_0h=3$ and $\mult_0h=4$.
   Each class will be divided into several classes according to the form of the initial
   term of $h$.
   Let $\ini h$ be the initial term of $h$, then it is a homogeneous polynomial
   of two variables.
   Therefore, $\ini h$ is presented as the product of linear forms.
   Now we divide all $f=x^2+h(y,z)$, that must be considered, into the following 
   8 classes:

\vskip.5truecm
   \noindent
   {\bf Class A} $\mult_0h=3$. The following $l$ and $l_i$  $(i=1,2,3)$  are linear forms
   with $l_i\neq l_j \ (i\neq j)$.      
  
     {(Class A-1)} $\ini h=l_1l_2l_3$. 
      {(Class A-2)} $\ini h={l_1}^2l_2$.
     {(Class A-3)} $\ini h=l^3$.
     
\vskip.5truecm
\noindent
   {\bf Class B} $\mult_0h=4$. The following $l$ and $l_i$  $(i=1,\ldots, 4)$  are linear forms with $l_i\neq l_j \ (i\neq j)$.   

    {(Class B-1)} $\ini h=l_1l_2l_3l_4$. 
      { (Class B-2)} $\ini h={l_1}^2l_2l_3$.
     {(Class B-3)} $\ini h={l_1}^2{l_2}^2$.

    {(Class B-4)} $\ini h= {l_1}^3l_2$.
   \ \ \ \ \  {(Class B-5)} $\ini h=l^4$.

\vskip.5truecm    
Our strategy for the proof of  $C_2$ is to show
 one of the following for each class  among (A-1)--(B-5):
\begin{enumerate}
\item[(i)]  $f$ is non-degenerate with respect to all faces of $\Gamma(f)$.
\item [(ii)]   $X$ has an isolated singularity at 0 and $f$ is non-degenerate with respect to all compact faces of 
the Newton polygon $\Gamma(f)$.
\item[(iii)]   We find a prime divisor $E$  over $A$ with center at 0 
such that  $a(E, A, I_X)=0$ and $\val_EI_X\leq n$ for some fixed $n$
and prove that $(A, I_X)$ is log canonical by constructing a log-resolution 
of $(A,I_X)$.
\item[(iv)] 
We  find a Newton polygon $\Gamma$ such that
$\mld(\Gamma)=-\infty$ and $\Gamma(f) \subset \Gamma$.
\end{enumerate}

  Indeed, if we prove one of the above (i)--(iv) for every $f$ described as in (\ref{f}),
  then the proof of Conjecture $C_2$ for Case 6 will be complete.
  Because if we prove (i) or (ii), we can apply Proposition \ref{non-deg}
  to get $C_2$.
  If we prove (iii), then it shows that 
  $E$ computes $\mld(0; A,I_X)=0$ and $\val_Ef\leq n$, i.e., $\nu(X, 0)=\mu(X,0)\leq n$.
  Then, 
  $C_2$ holds
  in this class.
  If we prove (iv), a toric divisor $E$ which computes $\mld(\Gamma)=-\infty$ also computes $\mld(0; A, I_X)=-\infty$.
  Therefore by Proposition \ref{non-deg} we obtain $C_2$ for this class.
  Now we start to pursue the strategy.
 
 \vskip.5truecm
 \noindent 
{(Class A-1)} $\ini h=l_1l_2l_3$.  
 
   In this case, we will show (ii).
   First we can see that $h$ is reduced, since the initial term of $h$ is reduced.
   Therefore $X$ has an isolated singularity at 0.
   Next we will show the non-degeneracy  of $f$ with respect to the compact faces 
   of the Newton polygon $\Gamma(f)$.
   
   By a coordinates transformation in $k[[y,z]]$ we may assume that
   $$l_1=y,\ \  l_2=z,\ \  l_3=y+z.$$
   Then, looking at the Newton polygon $\Gamma(f)$,
   we can see that a compact face $\sigma$ of $\Gamma(f)$ is either 
   a compact face $\tau$ of $\Gamma(h)$ or the convex hull
   $\sigma=\lla(2,0,0), \tau\rra$ generated by $(2,0,0)$ and a compact face  $\tau$  of 
   $\Gamma(h)$.
   Here, we denote the convex hull of the set $S$ by $\lla S\rra$.
   
   If a compact face $\sigma$ is of type $\lla(2,0,0), \tau\rra$,
   then $f$ is non-degenerate with respect to $\sigma$.
   Because in this case, $f_\sigma$ is represented as
    $f_\sigma= x^2+ h_\tau(y,z)$ and the singular locus of the hypersurface defined by 
    $f_\sigma$ must be in
   the zero locus of $x$ (here, we use the assumption that
   $char k\neq 2$).
   Therefore, the rest of the faces   which we should check the
   non-degeneracy of $f$ are the compact faces $\tau$ of $\Gamma(h)$.
   (This argument will work for all classes in Class A and B.)
   
   Now we check the non-degeneracy of $f$ with respect to the compact faces
   of $\Gamma(h)$.
   The compact face generated by $\ini h=yz(y+z)$ is $$\gamma=\lla(2,1), (1,2)\rra
   \subset \Gamma(h)\subset \RR^2$$
   and $f_\gamma=h_\gamma=yz(y+z)$ is clearly non-degenerate.
   Here, we list the other possible compact faces of $\Gamma(h)$ and check
   the non-degeneracy of $f$ there.
   
\begin{enumerate} 
   \item[$\bullet$] $\tau_1=\lla(2,1), (a, 0)\rra \ \ (a\geq 4), \ f_{\tau_1}=h_{\tau_1}=
   y^2(z- \alpha y^{a-2})$, $(\alpha\in k)$
   
    \item[$\bullet$] $\tau_2=\lla(0, b), (1,2)\rra \ \ (b\geq 4), \ f_{\tau_2}=h_{\tau_2}=
   z^2(y- \beta z^{b-2})$,  $(\beta\in k)$
\end{enumerate}   
   
     By the form of $f_{\tau_i}$, it is clear that $f$ is non-degenerate with respect 
     to the face $\tau_i$ $(i=1,2)$.
     This completes the proof of the fact that $f$ is non-degenerate with respect to
     all compact faces of $\Gamma(f)$, which yields the proof of (ii).
     In this case, by the formula  in  Proposition \ref{toricformula},
     we obtain $\mldmj(0; X)=1$ and the prime divisor $E_\pp$ ($\pp=(3,2,2)$)
     computes it.
     As $\val_ {E_\pp}f=\langle \pp, \Gamma(f)\rangle=6$, we obtain $\nu(X,0)\leq 6$ 
     and $k_{E_\pp}=6$
\vskip.5truecm
\noindent
     {(Class A-2)} $\ini h={l_1}^2l_2$.
     
     In this case, we will prove that either (i) or (ii) holds and $\mldmj(0; X)=1$ and
     $\nu(X,0)\leq 6$.
      By a coordinate change, we may assume that $\l_1=y $ and $l_2=z$.
     
     First we consider the case that $h$ is reduced.
     In this case $f$ has an isolated singularity at 0.
     Accoding to the argument as in A-1, 
     we have only to check the non-degeneracy of $f$ with respect to the compact faces
   of $\Gamma(h)$.
   The compact face generated by $\ini h=y^2z$ is $$\gamma=\lla(2,1)\rra
   \subset \Gamma(h)\subset \RR^2$$
   and $f_\gamma=h_\gamma=y^2z$ which is clearly non-degenerate.
   Here, we list the other possible compact faces of $\Gamma(h)$ and check
   the non-degeneracy of $f$ there.
   
\begin{enumerate} 
   \item[$\bullet$] $\tau_1=\lla (a, 0), (2,1)\rra \ \ (a\geq 4), \ f_{\tau_1}=h_{\tau_1}=
   y^2(z- \alpha y^{a-2})$, $(\alpha\in k)$
   
    \item[$\bullet$] $\tau_2=\lla(2,1), (1, b)\rra \ \ (b\geq 3), \ f_{\tau_2}=h_{\tau_2}=
   yz(y- \beta z^{b-2})$,  $(\beta\in k)$
   
   \item[$\bullet$] $\tau_3=\lla (1, b), (0, c) \rra \ \ (c\geq b+2), \ f_{\tau_3}=h_{\tau_3}=
   z^b(y- \lambda z^{c-b})$,  $(\lambda\in k)$
   
   \item[$\bullet$] $\tau'_2=\lla(2,1), (0,d)\rra \ \ (d\geq 4), \ f_{\tau'_2}=h_{\tau'_2}=
   z(y^2- \mu z^{c-1})$,  $(\mu\in k)$
   \end{enumerate}   
    By the form of $f_{\tau_i}$ and $f_{\tau'_2}$, it is clear that $f$ is non-degenerate with respect 
     to each face.
     Here, we note that we use $char k\neq 2$ only for the proof of $\tau'_2$.
     This completes the proof of the fact that $f$ is non-degenerate with respect to
     all compact faces of $\Gamma(f)$, which yields the proof of (ii).
     In this case, by the formula Proposition \ref{toricformula},
     we obtain $\mldmj(0; X)=1$ and the same prime divisor $E_\pp$ ($\pp=(3,2,2)$)
     as in A-1
     computes it.
     As $\val_ {E_\pp}f=6$, we obtain $\nu(X,0)\leq 6$.   
     
     Next we consider the case that $h$ is not reduced.
     In this case, $h$ is decomposed as
     $$h=h_1^2\cdot h_2,$$
     where $\ini h_1=y$ and $\ini h_2=z $.
     Then by a coordinate change in $k[[y,z]]$,
      we can put $h_1=y$ and   $h_2=z$.
      Hence we obtain $$f=x^2+y^2z,$$   
      which gives that $(X,0)$ is the pinch point and we already know in \cite{ir2} that
      $\mldmj(0; X)=1$ and $\nu(X,0)\leq 6$.

\vskip.5truecm
\noindent
          {(Class A-3)} $\ini h=l^3$.
          
          Under this situation, we will show (iii) in some cases, (iv) in some of the other
          cases and reduce to the case A-1 and A-2 in the rest of the cases.
        We may assume that $l=y$.
 
 \noindent
 (A-3-1)       
        First, if\  $\Gamma(h)\subset \Gamma((3,0),(0,7))$, then
        $\Gamma(f)$ is contained in $ \Gamma((2,0,0),(0,3,0),(0,0,7))$ generated by
         $(2,0,0),(0,3,0),(0,0,7)$ in Example \ref{example}, (2).
         This shows that $\mldmj(0; X)=-\infty$ and it is computed by the prime divisor
        $E_\pp$, where 
       $\pp=(21,14,6)\in N$, therefore we obtain   $ \nu(X,0)\leq 42.$
       In this case $k_{E_\pp}=40$, and therefore by Proposition \ref{C=U} it follows $b(E_\pp)\leq 39$.

       Now we may assume that
        there is an integer point $P$ on the boundary of $\Gamma(h)$
        such that $P\not\in \Gamma((3,0),(0,7))$.
        Then the possible coordinates of $P$ are
 \begin{equation}\label{points}       
        (2,2), (1,3), (1,4), (0,4), (0,5) \ \ \mbox{ and}\ \ (0,6).
\end{equation}         

\noindent
(A-3-2)        
        Assume that $h$ is reduced, then $X$ has an isolated singularity at 0.
        
        When, in particular either $(0,4)$ or $(0,5)$ is on the boundary of $\Gamma(h)$,
        then  every other  point in the list in (\ref{points}) is not on the boundary and
        $$f=x^2+\alpha y^3+\beta z^i + (\mbox{higher\ term})\ \ \ \ (i=4,5,\ \alpha, \beta\in \CC)$$ is non-degenerate 
        with respect to the compact faces of 
        $\Gamma(f)$ and $\mldmj(0; X)=1$ and this case we already know that
        $\nu(X,0)\leq 6$.  
        
        Next, when $(1,3)$ is on the boundary of $\Gamma(h)$,
        then every other  point in the list (\ref{points}) is not on the boundary and 
        $f=x^2+\alpha y^3+\beta yz^3 + (\mbox{higher\ term})$  $(\alpha, \beta\in \CC)$ is non-degenerate 
        with respect to the compact faces of 
        $\Gamma(f)$ and $\mldmj(0; X)=1$ and this case we already know that
        $\nu(X,0)\leq 6$.

       Next, note that the remaining points $(2,2), (1,4)$ and $(0, 6)$
        are lying on the segment connecting  $(3,0)$ and $(0,6)$.
        If some of these three points are lying on the boundary of $\Gamma(h)$,
        denote the face  generated by $(3,0)$ and these points by $\gamma$.
        Then, decompose $h$ as follows:
        $$h=h_\gamma +h',$$
        where $\Gamma(h')\subset \Gamma((3,0),(0,7))$.
        Here, $h_\gamma$ is a homogeneous polynomial of degree 3 
        in the variables $y$ and $Z=z^2$.
        Therefore it is decomposed into the products of linear forms in $y$ and $ Z$ as follows:
        
       $$h_\gamma=L_1L_2L_3,\ \ \mbox{or}\  L_1^2L_2, \ \ \mbox{or}\ L^3.$$
       By an appropriate coordinate change, we may assume that
       $L_1=L=y$, $L_2=z^2$ and $L_3=y+z^2$.
       Then in the last case we have the expression:
       $$h=y^3+h'',$$
       where $\Gamma(h'')\subset \Gamma((3,0), (0,7))$.
       Therefore,  we can reduce this case to (A-3-1).
        
       In the first two cases for $h_\gamma$, we can see that
        $f$ is non-degenerate with respect to $\gamma$
       and also non-degenerate with respect to the other possible faces:
       
\begin{enumerate}
       \item[$\bullet$] $\tau_1=\lla(2,2), (1, a)\rra$,  $a\geq 5$, 
        $\ f_{\tau_1}=h_{\tau_1}= yz^2(y+\alpha z^{a-2})$.
        \item [$\bullet$]  $\tau_2=\lla (1, a), (0,b) \rra$,  $b-2\geq a\geq 5$,   $\ f_{\tau_2}=h_{\tau_2}=z^a(y+\beta z^{b-a})$.
        \item [$\bullet$]  $\tau_3=\lla (1, 4), (0,c) \rra$,  $c\geq 7$,   $\ f_{\tau_3}=h_{\tau_3}=   
        z^4(y+\lambda z^{c-4})$.
\end{enumerate}
    Therefore in this case $\mldmj(0; X)=0$ by the formula in Proposition \ref{toricformula} and
    this value is computed by the prime divisor $E_\pp$, where $\pp=(3,2,1)$.
    We have 
    $\nu(X,0)\leq 6$ and $k_{E_\pp}=5$.
     
(A-3-3) Assume that $h$ is not reduced.
    
    There are two possibilities: $h=h_1^3$ and $h=h_1^2h_2$.
    In both cases $\ini h_1=\ini h_2=y$.
    Therefore, by the coordinate change, we may assume that $h_1=y$ in both cases.
    Then, in the first case we have:
    $$f=x^2+y^3,$$
    which implies $\Gamma(f)\subset \Gamma((2,0,0),(0,3,0),(0,0,7))$, which
    can be reduced to the case (A-3-1).
   
   While, in the second case we have 
   $$f=x^2+y^2(\alpha y+h'_2),\ \ (\alpha\in k)$$
   where $\mult_0h'_2\geq 2$.
   Here, if $h'_2$ does not contain the monomial $z^2$ as a summand,
   then $\Gamma(f)\subset \Gamma((2,0,0),(0,3,0),(0,0,7))$, which is again
   reduced to the case (A-3-1).
   
   When $h_2$ contains the monomial $z^2$ as a summand,
   then the singular locus $\sing(X)$ is defined by $x=y=0$.
   Let $A'\to A$ be the blow-up with the center $\sing(X)$ and 
  then let $A''\to A'$ be the blow-up with the center at 
  the origin $0'\in \spec k[y, x/y, z]\subset A'$.
   Then the composite $A''\to A'\to A$ becomes a log-resolution of $(A, I_X)$
   and the log-dicrepancies $a(E; A, I_X)\geq 0$ for every prime divisor $E$
   appearing on $A''$.
   Hence $(A, I_X)$ is log canonical, i.e., $X$ is MJ-log canonical.
   On the other hand we have
   $a(E_{(3,2,1)},A, I_X)=0$, which implies that the prime divisor 
   $E_{(3,2,1)}$ computes  $\mldmj(0; X)=0$,
   therefore $\nu(X,0)\leq 6$ and $k_{E_\pp}=5$.

\vskip.5truecm
\noindent
    {(Class B-1)} $\ini h=l_1l_2l_3l_4$.  
    
    In this case $h$ is reduced, and therefore $X$ has an isolated singularity at 0.
    By a coordinate change, we may assume that $l_1=y, l_2=z, l_3=y+z, l_4=y-z$.
    Then, we can see that $h$ is non-degenerate with respect to the face $\gamma$
    corresponding to $\ini h$.
    On the other hand, also with respect to the other possible faces
    $\tau_1=\lla(a,0),(3,1)\rra$ and $\tau_2=\lla(1,3),(0,b)\rra$,
    $h$ is non-degenerate.
    This can be checked in the same way as in (A-1).
    Therefore, by the formula in Proposition \ref{toricformula}, we obtain 
    $\mldmj(0;X)=0$ and the prime divisor $E_{(2,1,1)}$ computes it.
    Hence, $\nu(X,0)\leq 4$.

\vskip.5truecm
\noindent         
      {(Class B-2)} $\ini h={l_1}^2l_2l_3$.
       In this case, by a coordinate transformation we may assume that
       $l_1=y,\ \ l_2=z,\ \ l_3=y+z$.
            
\noindent
      (B-2-1) Assume that $h$ is reduced.  
      Then $X$ has an isolated singularity at 0.  
       Let $\gamma$ be the compact face corresponding to $\ini h$,
       then $\gamma=\lla(3,1),(2,2)\rra$.
       We can see that $h$ is non-degenerate with respect to $\gamma$,
       as $h_\gamma=\ini h=y^2z(y+z)$.
       On the other hand, also with respect to the other possible faces
\begin{enumerate}
\item[$\bullet$]       
    $\tau_1=\lla(a,0),(3,1)\rra$ $(a\geq 5)$, 
\item[$\bullet$]    
    $\tau_2=\lla(2,2),(1,b)\rra$ $(b\geq 4)$,  
\item[$\bullet$]    
    $\tau_3=\lla(1,b),(0, c)\rra$ $(c\geq 6)$ and 
\item[$\bullet$]    
    $\tau'_2=\lla(2,2), (0,d)\rra$ $(d\geq 5)$, 
\end{enumerate}    
    $h$ is non-degenerate.
   This can be proved in the same way as in (A-2).
   Here, we note that we use $char k\neq 2$ for the proof of non-degeneracy with
   respect to $\tau_2'$.
   In this case we also have $\mldmj(0;X)=0$ and the prime divisor $E_{(2,1,1)}$ computes it.
    Hence, $\nu(X,0)\leq 4$.
    
    (B-2-2) Assume that $h$ is not reduced.
    Then, by a coordinate transformation of $k[[y,z]]$, we can take 
    $h=h_1^2h_2$, such that $h_1=y$ and $\ini h_2=z(y+z)$.
    In this case $\sing(X)$ is defined by $x=y=0$.
    As in (A-3-3),   let $\varphi:A'\to A$ be the blow-up with the center $\sing(X)$, 
   then the proper transform $ X'\subset A'$ of $X$
   has an isolated singularity at a point, say $0'\in A'$.
   Compose $\varphi$ with the blow-up $A''\to A'$ with the center $0'$,
    then $A''\to A'\to A$ becomes a log-resolution of $(A, I_X)$.
   The log-discrepancies $a(E; A, I_X)\geq 0$ for every prime divisor $E$
   appearing on $A''$.
   Hence $(A, I_X)$ is log canonical, i.e., $X$ is MJ-log canonical.
   On the other hand we have
   $a(E_{(2,1,1)},A, I_X)=0$, which implies that the prime divisor 
   $E_{(2,1,1)}$ compute the $\mldmj(0; X)=0$,
   therefore $\nu(X,0)\leq 4$.

\vskip.5truecm
\noindent      
     {(Class B-3)} $\ini h={l_1}^2{l_2}^2$.
     By a coordinate change, we may assume that $l_1=y$ and $l_2=z$.

(B-3-1)  Assume that $h$ is reduced, then $X$ has an isolated singularity at 0.
The possible compact faces of $\Gamma(h)$ are:

   the same $\tau_2, \tau_3$ and $\tau_2'$ as in (B-2-1) and 
   
   the symmetric faces $\gamma_2,\gamma_3$ and
   $\gamma_2'$ of them with respect to $y$ and $z$.
   
   Therefore, $f$ is non-degenerate with respect to all compact faces of $\Gamma(f)$
   and $\mldmj(0;X)=0$ with $E_{(2,1,1)}$ as a computing prime divisor.
   Hence, $\nu(X,0)\leq 4$.

 (B-3-2)
     Assume that $h$ is not reduced.
    In this case, there are two possibilities: 
    $h=h_1^2h_2$, where  $h_2$ is reduced,
    and $h=h_1^2h_2^2$.
    
    In the first case, $h=h_1^2h_2$, where  $h_2$ is reduced,
    we can put $h_1=y$  and $\ini h_2=z^2$ by the coordinate change,
    as we may assume that $\ini h_1=y$.
    In this case the singular locus $\sing(X)$ of $X$ is defined by
    $x=y=0$.
    Let $\varphi:A^{(1)}\to A$ be the blow-up of $A$ with the center $\sing(X)$.
    Then the proper transform $X^{(1)}$ of $X$ in $A^{(1)}$ has an isolated singularity
    at a point $0_1$ which is $A_n$-singularity.
    The composite of the successive blow-ups at the singularities
    and $\varphi$:
    $$A^{(m)}\to A^{(m-1)}\to \cdots A^{(1)}\to A,$$
    gives a log-resolution of $(A, I_X)$.
    Here, we observe that every exceptional divisor $E$ has log-discrepany
    $a(E; A,I_X)=0$, therefore $\mldmj(0;X)=0$ and $E_2$ computes it.
    As $\val_{E_2}I_X=4$, we can see that $\nu(X,0)\leq 4$.
    
    In the second case, $h=h_1^2h_2^2$,
    we can put $h_1=y$ and $h_2=z$ by the coordinate change,
    as we may assume that $\ini h_1=y$ and $\ini h_2=z$.
    Hence, we obtain 
    $$f=x^2+y^2z^2$$
    which is non-degenerate with respect to all faces of $\Gamma(f)$.
    In this case $\mldmj(0;X)=0$ and the divisor $E_{(2,1,1)}$ computes it.
    Therefore, $\nu(X,0)\leq 4$.

\vskip.5truecm
\noindent     
    {(Class B-4)} $\ini h= {l_1}^3l_2$.
    By the coordinate change, we may assume that $l_1=y$ and $l_2=z$.
    Then, $$\Gamma(f)\subset \Gamma((2,0,0), (0,3,1), (0,0,5)),$$
    which yields that $\mldmj(0;X)=-\infty$ and the prime divisor 
    $E_{(15,8,6)}$ computes it and  
  $ \nu(X,0)\leq 30$  by Example \ref{example}, (3).
    
\vskip.5truecm
\noindent    
    {(Class B-5)} $\ini h=l^4$. 
     By the coordinate change, we may assume that $l=y$.
     Then  $f$ is of the form:
     $$f=x^2+y^4+(\mbox{terms\ of\ degree}\ \geq 5).$$
     Then, $$\Gamma(f)\subset \Gamma((2,0,0), (0,0, 4), (0,0,5)),$$
    which yields that $\mldmj(0;X)=-\infty$ and the prime divisor 
    $E_{(10,5,4)}$ computes it and  
  $ \nu(X,0)\leq 20$  by Example \ref{example}, (4).
  
  Now we  obtain $\nu(X,0)$ for all cases.
  In each case, we can calculate also the bounds of $k_E$ and
  $b(E)$ for a prime divisor $E$ computing $\mldmj(x;X)$.
  As conclusions, $\nu(X, 0)=\mu(X,0)\leq 42$, $k_E\leq 58$ and $b(E)\leq 39$
  for all $(X,0)$ and a prime divisor $E$ computing $\mldmj(x;X)$.
\end{proof}

 \begin{cor}Assume the characteristic of the base field $k$ is not $2$.
Then, 
Conjecture $C_2$ holds  in the category of normal locally complete intersection singularities of dimension $2$  over $k$.

In particular, 
for every  singularity $(X,x)$ in this category 
there is a prime divisor $E$ over $X$  computing $\mld(x;X)(=\mldmj(x;X))$ such that 
$b(E)\leq 20$.
\end{cor}

\begin{proof}
    The first statement follows from Theorem \ref{surface}, 
    because the equality 
    $\mld(x;X)=\mldmj(x;X)$ holds for locally a complete intersection singularity $(X,x)$.
   For the second statement, 
   let $X\subset A$ be a closed immersion into a non-singular variety $A$ of dimension $N=
     \emb(X,x)$.
 Let 
 \begin{equation}
\label{sequence}     
    A^{(n)}\stackrel{\varphi_n}\longrightarrow A^{(n-1)}\to\cdots\to
A^{(1)}\stackrel{\varphi_1}\longrightarrow A^{(0)}=A
\end{equation}
be the minimal sequence of blow-ups to obtain a prime divisor $\he\subset A^{(n)}$
     computing $\mldmj(x;X)$
     and let $\he_i$ be the exceptional divisor dominant to the center of $\varphi_i$ 
     and $\he^{(n)}=\he$.
    
     We already know the following:
\begin{enumerate}      
    \item[(i)] If $N\geq 5$, then $n=1$ by Proposition \ref{C=U}, 
     \item[(ii)] if $N=4$, then $k_\he\leq 58$ and $n\leq 29$ (cf. Case 3 in the proof of 
     Theorem \ref{surface}), and 
     \item[(iii)] if $N=3$, then $k_\he\leq 40$ and $n\leq 39$ (cf. Case 6, A-3-1). 
 \end{enumerate}

     \noindent
{\bf Claim.} For an irreducible component $E^{(n)}$ of $\he^{(n)}\cap X^{(n)}$
  there exists a prime divisor $E$ over $X$ such that $E$ computes $\mld(x;X)$ and 
  has the center $E^{(n)}$ on $X^{(n)}$.
  
   Once the claim is proved, then the required statement of the corollary follows.
   Indeed, first we know $b(E)\leq n=b(\he)$, by the definition of $b(E)$.
   Then the  sequence (\ref{sequence}) consists of 
\begin{enumerate}
   \item[$\bullet$] $b(E)$-times blow-ups with 0-dimensional centers and
   \item[$\bullet$] $(b(\he)-b(E))$-times blow-ups with 1-dimensional centers.
\end{enumerate}   
    Therefore, it follows
    $$k_\he\geq (N-2)(b(\he)-b(E))+(N-1)b(E),\  \mbox{ hence\ we \ have}$$
    $$b(E)\leq \frac{k_\he}{N-1}.$$
    Now in the case (i), we have $b(E)=1$.
    In the case (ii) and (iii), we have 
    $$b(E)\leq \frac{58}{3} \ \mbox{and}\ \ b(E)\leq \frac{40}{2}\ \mbox{respectively},$$
    which shows the required statement.

       Now we are going to prove the claim.
       Let $\varphi_{n+1}:A^{(n+1)}\to A^{(n)}$ be the blow-up with the center $E^{(n)} $ which is contained in 
       $ X^{(n)}$.
       Let $p_n\in E^{(n)}$ be the generic point and let $\he^{(n+1)}$ be the exceptional divisor of 
       $\varphi_{n+1}$ dominating $E^{(n)}$.
       Then, we have
\begin{equation}\label{sandwich}
 \begin{array}{ll}
\mld(x;A, I_X)& \leq k_{\he^{(n+1)}}-c\cdot\val_{\he^{(n+1)}}I_X +1\\
  & \leq (k_{\he^{(n)}}+c)-c\cdot\val_{\he^{(n)}}I_X-c\cdot\mult(X^{(n)}, p_n)+1\\
 &\leq k_{\he^{(n)}}-c\cdot\val_{\he^{(n)}}I_X+1 = \mld(x;A, I_X).\\
 \end{array}
 \end{equation}

\noindent     
   Here, the middle inequality follows from $\codim(E^{(n)}, A^{(n)})=c+1$.
   Hence, we obtain that $\he^{(n+1)}$ also computes $\mld(x;A,I_X)$.  
  
   Now, let 
\begin{equation}
\label{sequence2}     
    A^{(m)}\stackrel{\varphi_m}\longrightarrow A^{(m-1)}\to\cdots\to
A^{(n+1)}\stackrel{\varphi_{n+1}}\longrightarrow A^{(n)}
\end{equation}
be the sequence of blow-ups so that $X_m$ is non-singular and crossing $\he^{(m)}$ normally
at the generic point $p_m$ of an irreducible component $E^{(m)}$  of $\he^{(m)}\cap X_m$,
where $E^{(m)}$ is dominant to $E^{(n)}$.
   Then, applying the same discussion as in (\ref{sandwich}) to each blow-up of the sequence (\ref{sequence2}),
   we obtain that $\he^{(m)}$ computes $\mld(x;A,I_X)$. Now we have
   $$\mld(x;X)\leq k_{E^{(m)}}+1\leq k_{\he^{(m)}}+1-c\cdot\val_{\he^{(m)}}I_X=\mld(x;A,I_X).$$
   Therefore $E^{(m)}$ is a required prime divisor over $X$ in the claim.
    \end{proof}

\begin{rem} The minimal value $\widetilde b(E)$ such that $E$ computes $\mldmj(x;X)$
  is not bounded for all locally complete intersection singularities.
  Actually for a singularity $(X,0)\subset \AA^2$ defined by $x^2-y^m=0$ for odd $m$.
  Then, in order to obtain a  variety normal at the generic point of the prime divisor
  computing $\mldmj(0;X)=-\infty$,
  the necessary number of blow-ups tends to infinity, when $m\to \infty$.

\end{rem}

\begin{rem} In the theorem for surfaces we assume that $char k\neq 2$. 
    The only case we assume this condition is for hypersurface double points.
    The proof of $C_2$ for $char k=2$ will be treated in another paper.
    This is because we should take care of more cases for $char k=2$ than considered here and the volume of the proof may exceed the capacity of a paper of a reasonable size.

\end{rem}

\begin{rem} In this paper, we concentrate only on the singularities of $X$, i.e., the singularity of the trivial pair $(X, \o_X)$.
  Because it is the skeleton of the structure and seeing this first would help the further work on
  singularities of general pairs $(X, \a^n)$ .

\end{rem}

\vskip1truecm

\noindent Shihoko Ishii, \\ Department of Mathematics, Tokyo Woman's Christian University,\\
2-6-1 Zenpukuji, Suginami, 167-8585 Tokyo, Japan.\\

\end{document}